\documentclass[a4paper,12pt,twoside]{amsart}

\usepackage{amsfonts}
\usepackage{amsmath}
\usepackage{amssymb}
\usepackage{pstricks,pstricks-add,pst-math,pst-xkey}

\newcommand{\C}{\mathbbm{C}}
\newcommand{\R}{\mathbbm{R}}
\newcommand{\Z}{\mathbbm{Z}}

\newcommand{\g}{\mathfrak{g}}

\newcommand{\lt}{\mathfrak{t}}

\newcommand{\f}{\mathfrak{f}}

\newcommand{\so}{\mathfrak{so}}
\newcommand{\lsl}{\mathfrak{sl}}

\newcommand{\SO}{\mathrm{SO}}

\newcommand{\End}{\mathrm{End}}

\newtheorem{thm}{Theorem}[section]
\newtheorem{lem}[thm]{Lemma}
\newtheorem{prop}[thm]{Proposition}

\newtheorem{defn}[thm]{Definition}
\newtheorem{rem}[thm]{Remark}
\newtheorem{eg}[thm]{Example}

\usepackage{bbm}

\title[Harmonic maps and canonical elements]{Harmonic maps of  finite uniton number and their canonical elements}

\author{Nuno Correia}
\address{Universidade da Beira Interior\\
Rua Marques d'Avila e Bolama, 6200-001 Covilha, Portugal}
\email{ncorreia@ubi.pt}

\author{Rui Pacheco}
\address{Universidade da Beira Interior\\
Rua Marques d'Avila e Bolama, 6200-001 Covilha, Portugal}
\email{rpacheco@ubi.pt}

\thanks{The authors were partially supported by the Portuguese Government through FCT, under the project PEst-OE/MAT/UI0212/2011 (CMUBI)}

\begin{document}

\begin{abstract}
We classify all harmonic maps with finite uniton number  from a Riemann surface into an arbitrary compact simple Lie group $G$, whether $G$ has trivial centre or not, in terms of certain pieces of the Bruhat decomposition of the group  $\Omega_\mathrm{alg}{G}$ of algebraic loops in $G$ and corresponding canonical elements. This will allow us to give estimations for the minimal uniton number of the corresponding harmonic maps  with respect to different representations and to make more explicit the relation between  previous work by different authors on harmonic two-spheres in classical compact Lie groups and their inner symmetric spaces and the Morse theoretic approach to the study of such harmonic two-spheres introduced by Burstall and Guest. As an application, we will also give some explicit descriptions of harmonic spheres in low dimensional spin groups making use of spinor  representations.
\end{abstract}

\maketitle

\section{Introduction}
The conditions of harmonicity for  maps from a Riemann surface $M$ into a compact Lie group $G$ amount to the flatness of a
$S^1$-family of connections  \cite{uhlenbeck_1989}. On simply connected domains, this zero-curvature formulation establishes a correspondence between harmonic maps $M\to G$ and certain holomorphic maps $M\to\Omega G$, the \emph{extended solutions}, into the group
of based smooth loops in ${G}$.
The simplest situation occurs when the
Fourier series associated to an extended solution has finitely
many terms: in this case, the extended solution and the corresponding harmonic map are said
to have \textit{finite uniton number}. It is well known that all harmonic maps from the two-sphere to a compact Lie group have finite uniton number \cite{uhlenbeck_1989}.

When ${G}$ is a compact simple Lie group with  trivial centre and  $n$-dimensional maximal torus, Burstall and Guest \cite{burstall_guest_1997} have classified harmonic maps with finite uniton number from $M$ into $G$ and into its inner symmetric spaces $G/H$ in terms of the $2^n$ \emph{canonical elements} of a certain integer lattice $\mathfrak{I}({G})$ in $\g$, the Lie algebra of $G$. One of the  main ingredients in that paper is the Bruhat decompostion
of the group of algebraic loops $\Omega_\mathrm{alg}{G}$.  Each piece of this decomposition corresponds to an element of
$\mathfrak{I}({G})$ and coincides with some unstable manifold of  the energy flow on the K\"{a}hler manifold $\Omega_\mathrm{alg}{G}$.

Eschenburg, Mare and  Quast \cite{EMQ} extended  Burstall and Guest results to outer symmetric spaces, and Correia and Pacheco \cite{correia_pacheco_4}  extended the notion of canonical elements to groups with non-trivial centre and investigated
 the canonical elements of $SU(n+1)$, which has centre isomorphic to the cyclic group $\mathbbm{Z}_n$. It is the purpose of the present paper to describe all the canonical elements of the remaining compact simple Lie groups. This will allow us to give estimations for the \emph{minimal uniton number} of the corresponding harmonic maps  with respect to different representations  and to make more explicit the relation between the ideas of Burstall and Guest in \cite{burstall_guest_1997} and several other papers \cite{Ai_1986,Bahy_Wood_1989,BW_1991,eells_wood_1983,pacheco_2006,svensson_wood_2010} on harmonic maps with finite uniton number into classical Lie groups and their inner symmetric spaces.
For example, the second author has exploited in \cite{pacheco_2006} the Grassmannian theoretic point of view
introduced by Segal \cite{segal_1989} in order to study harmonic maps from a
two-sphere into the compact symplectic group $Sp(n)$ (which has centre isomorphic to $\mathbbm{Z}_2$): a ``uniton factorization" for harmonic two-spheres in $Sp(n)$ and
 a characterization of harmonic two-spheres in
 $\mathbbm{H}P^{n}$ were given. Such a characterization provides an alternative
 geometrical approach to the work of  Bahy-El-Dien and Wood \cite{BW_1991} and generalizes
the work of Aithal \cite{Ai_1986} on  harmonic
 two-spheres in $\mathbbm{H}P^{2}$. In Section \ref{cn} of the present paper we will be able to describe that ``uniton factorization" in terms of canonical elements and to identify the canonical element of $Sp(n)$ associated to each harmonic map with finite uniton number into $\mathbbm{H}P^{n}$.

The spin groups  have non-trivial centre and are the double covers of the special orthogonal groups. The spinor representation of  $Spin(m)$, which does not descend to a representation of $SO(m)$, induces an  embedding of $Spin(m)$ in  some classical Lie group. For low dimensions, such embeddings are well described in the literature  \cite{bryant,fulton_harris,harvey_lawson}. Since $Spin(3)\cong SU(2)$ and $Spin(6)\cong SU(4)$, we shall be able to compare the results in the present paper with those in \cite{correia_pacheco_4}. In  Section \ref{spinorharmonic}  we will make use of the characterization of $Spin(7)$ as the subgroup of $SO(8)$ preserving the \emph{triple cross product} in the division algebra of the octonions in order to give some explicit constructions of harmonic maps of finite uniton number into $Spin(7)$.

\section{Representation of a Lie algebra $\g$}
In this section we fix some notation and recall standard facts concerning simple Lie algebras and their representations. For details see \cite{fulton_harris}.

Let $\g^\C$ be a complex simple Lie algebra, with compact form $\mathfrak{g}$, and $\lt\subset\g$ be a maximal torus. Consider the corresponding
set of roots $\Delta\subset \sqrt{-1} \lt^*$ and denote by $\g_\alpha$ the root space of  $\alpha\in \Delta$. Fix a Weyl chamber, which is equivalent to fix  a basis  $\alpha_1,\ldots,\alpha_k\in\sqrt{-1} \lt^*$ of positive simple roots, with dual basis $H_1,\ldots, H_k\in\mathfrak{t}$, that is $\alpha_i(H_j)=\sqrt{-1}\,\delta_{ij}$. Denote by  $\omega_1,\ldots,\omega_k\in \sqrt{-1} \lt^* $ the fundamental weights of $\mathfrak{g}$, with dual basis $\eta_1,\ldots,\eta_k\in \mathfrak{t}$. The positive simple roots and the fundamental weights are related by  the Cartan matrix $\mathcal{C}$: $\left[\alpha_i\right]=\mathcal{C}\left[\omega_i\right]$, and consequently  $\left[H_i\right]=(\mathcal{C}^T)^{-1}\left[\eta_i\right]$. The matrix $\mathcal{C}$ has integer entries whereas $(\mathcal{C}^T)^{-1}$ has rational positive entries.

Let $\tilde{G}$ be the simply connected Lie group with Lie algebra $\mathfrak{g}$ and  $C$ be a (discrete) subgroup of the centre $Z(\tilde G)$ of $\tilde G$. Denote by $\exp_G$ the exponential map of $G=\tilde{G}/C$. If $C=\{e\}$ we simply write $\exp=\exp_{\tilde G}$.
Lie algebra representations $\rho:\g\to\mathrm{End}(V)$ correspond to Lie group  representations ${\rho}:\tilde G\to\mathrm{End}(V)$ by
${\rho}(\exp(X))=\exp(\rho(X))$,
 for all $X\in\g$.
Denote by $V^{\omega}$  the irreducible representation of the Lie algebra $\g$ with highest weight $\omega$.
 \begin{thm}\cite{fulton_harris}\label{repG}  The representation $V^{\omega}$  is a representation of $G$ if and only if $\omega(X)\in2\pi\sqrt{-1}\Z$ whenever $\exp(X)\in C$.\end{thm}

Consider the integer lattices $\Lambda_R:=\Z H_1\oplus\ldots\oplus \Z H_k$ and   $\Lambda_W:=\Z\eta_1\oplus\ldots\oplus\Z\eta_k$.
 Define  $\mathfrak{I}(G)=(2\pi)^{-1} \exp_G^{-1}(e)\cap \lt$. These are related as follows.

 \begin{thm}\cite{fulton_harris}\label{lat}
   If $C=Z(\tilde G)$,  then  $\mathfrak{I}(G)=\Lambda_R$. If $C=\{e\}$, then $\mathfrak{I}(G)=\Lambda_W$. In the general case, we have
   $\Lambda_W\subseteq \mathfrak{I}(G)\subseteq \Lambda_R.$
 \end{thm}

Given $\xi\in \mathfrak{I}({G})$, denote by $\g^\xi_i$ the $\sqrt{-1}\,i$-eigenspace of $\mathrm{ad}{\xi}$, with $i\in\Z$:
 \begin{equation}\label{gis1}
 \g_i^\xi=\bigoplus_{\alpha(\xi)=\sqrt{-1}\,i}\g_\alpha.
\end{equation}

\section{The Bruhat Decomposition of $\Omega_{\mathrm{alg}}(G)$}

Next we describe the Bruhat decomposition for algebraic loop groups. For more details we refer the reader to \cite{burstall_guest_1997} and \cite{pressley_segal}.

Set $$\Lambda^+{G}^\mathbbm{C}=\{\gamma:S^1\to {G}^\mathbbm{C}\,|\,\, \text{$\gamma$ and $\gamma^{-1}$ extend holomorphically for }|\lambda|\leq 1\}.$$
Taking account the \emph{Iwasawa decomposition} \cite{pressley_segal} of
$\Lambda{G}^\mathbbm{C}=\{\gamma:S^1\to {G}^\mathbbm{C}\,|\,\, \text{$\gamma$ smooth}\},$  each $\gamma\in  \Lambda
G^\mathbbm{C}$ can be written uniquely in the form $\gamma=\gamma_R\gamma_+$, with $\gamma_R\in \Omega G$ and $\gamma_+\in \Lambda^+G^\mathbbm{C}$. Consequently, there exists a
natural action  of $\Lambda^+ G^\C$ on
$\Omega G$: if $g\in\Omega G$ and
$h\in \Lambda^+ G^\C$, then $h\cdot g=(hg)_R.$

Fix a maximal torus ${T}$ of $G=\tilde{G}/C$ with Lie algebra $\lt\subset\g$. The integer lattice $\mathfrak{I}(G)$ may be identified with the group of homomorphisms ${S}^1\to {T}$, by
associating to $\xi\in \mathfrak{I}({G})$ the homomorphism
$\gamma_\xi$ defined by $\gamma_\xi(\lambda)=\exp_G{(-\sqrt{-1}\ln(\lambda)\xi)}$.
For each
$\xi\in \mathfrak{I}({G})$, we write
$\Omega_\xi=\{g\gamma_\xi g^{-1} \,|\,\, g \in {G}\},$ the conjugacy class of homomorphisms ${S}^1\to {G}$ which contains $\gamma_\xi$.
This is a complex homogeneous space:
$$\Omega_\xi\cong G^\C\big/{P}_\xi,\,\text{ with } \,{P}_\xi=G^\C\cap \gamma_\xi\Lambda^+G^\C\gamma_\xi^{-1}.$$
The Lie algebra of the isotropy subgroup ${P}_\xi$ is precisely the parabolic subalgebra $\mathfrak{p}_\xi=\bigoplus_{i\leq 0}\g^\xi_i$ induced by $\xi$.

A based loop $\gamma\in \Omega G$ is \emph{algebraic} if $\gamma$ and $\gamma^{-1}$ have finite Fourier expansions under all (unitary) representations $\rho$ of $G$. We denote by $\Omega_{\mathrm{alg}}G$ the group of all algebraic loops in $G$.
Consider the positive set of roots $\Delta^+\subset \Delta$ and set $$\mathfrak{I}'(G)=\{\xi\in\mathfrak{I}(G)|\, \alpha(\xi)\geq 0 \text{ for all }\alpha\in \Delta^+\}.$$

The Bruhat Decomposition of $\Omega_{\mathrm{alg}}G$ is given by the following theorem.
\begin{thm}\cite{pressley_segal} $\Omega_\mathrm{alg}{G}=\bigcup_{\xi\in \mathfrak{I}'({G})}U_\xi(G)$, with $U_\xi(G):=\Lambda^+_{\mathrm{alg}}{G}^\C\cdot\gamma_\xi$.
\end{thm}
Each $U_\xi(G)$ is also a  complex homogeneous space
of $\Lambda^+_{\mathrm{alg}}{G}^\C$
with isotropy subgroup at $\gamma_\xi$ given by $\Lambda^+_{\mathrm{alg}}{G}^\C\cap \gamma_\xi \Lambda^+{G}^\C \gamma_\xi^{-1}.$ Moreover,  $U_\xi({G})$ carries a structure of holomorphic vector bundle over $\Omega_\xi$ whose bundle map
$u_\xi:U_\xi({G})\to \Omega_\xi$ is precisely the natural map
  $[\gamma]\mapsto [\gamma(0)]$.  

Following \cite{correia_pacheco_3}, consider the partial order $\preceq$ over $\mathfrak{I}'({G})$ defined by:
$\xi\preceq \xi'$ if $\mathfrak{p}^{\xi}_i\subseteq \mathfrak{p}^{\xi'}_i$
for all $i\geq 0$, where $\mathfrak{p}_i^\xi=\bigoplus_{j\leq i}\g_j^\xi$.
Given two elements $\xi,\xi'\in \mathfrak{I}'({G})$ such that $\xi\preceq \xi'$, it can be shown \cite{correia_pacheco_3} that
   $$\Lambda^+_{\mathrm{alg}}{G}^\C\cap \gamma_\xi \Lambda^+{G}^\C \gamma_\xi^{-1}\subseteq \Lambda^+_{\mathrm{alg}}{G}^\C\cap \gamma_{\xi'} \Lambda^+{G}^\C \gamma_{\xi'}^{-1}.$$
This allows us to define, for $\xi\preceq \xi'$, a $\Lambda^+_{\mathrm{alg}}{G}^\C$-invariant fibre bundle morphism
$\mathcal{U}_{\xi,\xi'}:U_\xi(G)\to U_{\xi'}(G)$  by
\begin{equation*}
\mathcal{U}_{\xi,\xi'}(\Psi\cdot\gamma_{\xi})=\Psi\cdot\gamma_{\xi'}, \quad \Psi\in\Lambda^+_{\mathrm{alg}}{G}^\C.\end{equation*}
Since the holomorphic structures on  $U_\xi({G})$ and $U_{\xi'}({G})$ are induced by the holomorphic structure on $\Lambda^+_{\mathrm{alg}}{G}^\C$, the fibre-bundle morphism  $\mathcal{U}_{\xi,\xi'}$ is holomorphic.

\section{Inner $G$-symmetric spaces}
Given a compact connected Lie group ${G}$, it is well-known that
 each connected component of $\sqrt{e}=\{g\in
{G}\,|\,\,g^2=e\}$ is a compact inner symmetric space totally geodesically  embedded in $G$.
Conversely, consider a compact  connected inner symmetric space $N=G/K$ with inner involution $\sigma=Ad(s_0)$, where $s_0\in G$.
 Let $C$ be any  subgroup of the (discrete) centre $Z(G)$ of $G$ and suppose that $s_0^2\in C$.
 Take a fixed point $p_0\in N$. Then $\iota:N\to \sqrt{e}\subset G/C$ defined by $\iota(g\cdot p_0)=[s_0\sigma(g^{-1})g]=[g^{-1}s_0g]$ is a totally geodesic immersion, the \emph{Cartan immersion} of $N$ in $\sqrt{e}\subset G/C $.
 If $G=I_0(N)$ is the identity connected component of the group of isometries of $N$, which is a compact Lie group with trivial centre, then the Cartan immersion of $N$ in  $G$ is given by $\iota(p)=s_p$, where $s_p$ is the geodesic symmetry at $p$.

Following \cite{burstall_guest_1997},  define the involution $\mathcal{I}:\Omega {G}  \rightarrow \Omega {G}$ by $\mathcal{I}(\gamma)(\lambda)
=\gamma(-\lambda)\gamma(-1)^{-1}.$ Write $$\Omega^\mathcal{I}
{G}=\{\gamma\in \Omega {G}\,|\,\,\mathcal{I}(\gamma)=\gamma\}$$
 for the fixed set of $\mathcal{I}$. We can associate to an element $\xi \in \mathfrak{I}'({G})$  the inner symmetric space
\begin{equation}\label{Nxi}
N_\xi=\{g\gamma_\xi(-1)g^{-1}\,|\,\,g\in G\}\subseteq \sqrt{e},
\end{equation}
whose dimension is given by
\begin{equation}\label{dimxi}
\dim N_\xi= \dim\!\!\!\bigoplus_{\text{$\alpha(\xi)$ odd}}\!\!\!\mathfrak{g}_\alpha.
\end{equation}

Observe that, for $\xi$ and $\xi'$ in  $\mathfrak{I}'(\mathrm{G})$, if $\xi-\xi'\in\mathfrak{I}^{2}({G}):=\pi^{-1}\exp_G^{-1}(e)\cap \mathfrak{t},$ then $N_\xi=N_{\xi'}$.
Moreover, as shown in \cite{correia_pacheco_3},
 if $\xi\preceq \xi'$, then $\mathcal{U}_{\xi,\xi'}(U_\xi^{\mathcal{I}}({G}))\subset U_{\xi'}^{\mathcal{I}}({G})$. To sum up, if we define a new partial order $\preceq_\mathcal{I}$ in $\mathfrak{I}'({G})$ by
 $\xi\preceq_\mathcal{I}\xi'$ if $\xi\preceq \xi'$ and $\xi-\xi'\in\mathfrak{I}^{2}({G})$,
 the following holds.
\begin{prop} If $\xi\preceq_\mathcal{I} \xi'$, then $\mathcal{U}_{\xi,\xi'}(U_\xi^{\mathcal{I}}({G}))\subset U_{\xi'}^{\mathcal{I}}({G})$ and $N_\xi=N_{\xi'}$. \end{prop}

\begin{rem} Taking account Theorem \ref{lat} we have: if $C=Z(\tilde G)$,  then  $\mathfrak{I}^2(G)=2\Lambda_R$; if $C=\{e\}$, then $\mathfrak{I}^2(G)=2\Lambda_W$; in the general case,
   $2\Lambda_W\subseteq \mathfrak{I}^2(G)\subseteq 2\Lambda_R.$
\end{rem}

\begin{rem}\label{togeodequivalent}
 Given $\xi,\xi'\in\mathfrak{I}'(G)$, if $\exp_G(\pi(\xi-\xi'))\in Z( G)\cap \sqrt {e}$, then $N_\xi\cong N_{\xi'}$ as totally geodesic submanifolds of $G$. In fact, denoting $g_0=\exp_G(\pi(\xi-\xi'))$, we have $N_\xi=g_0N_{\xi'}$.
\end{rem}

\section{Harmonic maps and Extended Solutions}

\subsection{Extended solutions}

Let $M$ be a Riemann surface and equip $G$ with  a bi-invariant metric. Recall from \cite{uhlenbeck_1989} that a harmonic map $\varphi:M\to G$ has \emph{finite uniton number} if it admits an extended solution $\Phi:M\to \Omega G$, with $\varphi=\Phi_{-1}$, and a constant loop $\gamma$ such that $\gamma \Phi:M\to \Omega_{\mathrm{alg}}G$. Since $\gamma\Phi$ is also an extended solution corresponding to the same harmonic map (up to isometry), we will assume that $\Phi$ takes values in $\Omega_{\mathrm{alg}}G$. In this case, given a representation $\rho:G\to \mathrm{Aut}(V)$,
we can write
$\rho\circ\Phi=\sum_{i=s}^r\zeta_i\lambda^i$ for some $s\leq r\in\mathbbm{Z}$ and $\zeta_i\in \mathrm{End}(V)$, with $\zeta_r,\zeta_s\neq 0$. The \emph{uniton number} of $\Phi$ with respect to the representation $\rho$ is the number
$r_\rho(\Phi)=r-s$. We also define the \emph{minimal uniton number} of $\varphi$ with respect to $\rho$ as the non-negative integer
$$r_\rho(\varphi)=\min\{r_\rho(\gamma\Phi)|\,\, \gamma\in\Omega_{\mathrm{alg}}G\}$$
and the \emph{minimal uniton number} of $G$ with respect to $\rho$ as the non-negative integer
$$r_\rho(G)=\max\{r_\rho(\varphi)|\, \textrm{$\varphi:M\to G$ harmonic has finite uniton number}\}.$$

\begin{rem}
  If the representation is an orthogonal representation, then we must have $\rho\circ\Phi=\sum_{i=-s}^s\zeta_i\lambda^i$, with $s\geq 0$ and $\zeta_s=\overline{\zeta}_{-s}\neq 0$. Burstall and Guest \cite{burstall_guest_1997} considered only the adjoint representation of Lie groups, which is an orthogonal representation, and defined the minimal uniton number of the extended solution $\Phi$ as the non-negative integer $s$. Hence the minimal uniton number of an extended solution with respect to the adjoint representation in the present paper is twice that of Burstall and Guest \cite{burstall_guest_1997}.
\end{rem}

Since
the embedding of each component of $\sqrt{e}$ in $G$ is totally
geodesic,  harmonic maps into inner symmetric spaces can be viewed as special harmonic maps into ${G}$.
Given an extended solution
 $\Phi:M\rightarrow \Omega^\mathcal{I} {G}$, the harmonic map
 $\varphi=\Phi_{-1}$   takes values in some
 connected component of $\sqrt{e}$.
Conversely, if $\varphi:M\rightarrow\sqrt{e}$ is a harmonic map and $M$ is simply connected,  there exists an extended
 solution $\Phi:M\rightarrow \Omega^\mathcal{I} {G}$ such that
 $\varphi=\Phi_{-1}$. See  \cite{burstall_guest_1997} for details.

Off a discrete subset $D$ of $M$, an extended solution $\Phi:M\to\Omega_{\mathrm{alg}}{G}$ take values in  $U_\xi({G})$ for some $\xi\in \mathfrak{I}'({G})$, as observed in \cite{burstall_guest_1997}. The bundle morphism $\mathcal{U}_{\xi,\xi'}$ and the bundle map $u_\xi$ are well behaved with respect to extended solutions.
\begin{thm}\label{popo}\cite{correia_pacheco_3} Given an extended solution $\Phi:M\setminus D\to U_\xi({G})$ and an element $\xi'\in \mathfrak{I}'({G})$ such that $\xi\preceq \xi'$, then
$\mathcal{U}_{\xi,\xi'}(\Phi)=\mathcal{U}_{\xi,\xi'}\circ \Phi:M\setminus D\to U_{\xi'}({G})$ is a new extended solution. \end{thm}
\begin{thm}\cite{burstall_guest_1997} If  $\Phi:M\setminus D\to U_\xi({G})$ is an extended solution, then  $u_\xi\circ\Phi:M\setminus D\to \Omega_\xi$ is an extended solution. \end{thm}
Extended solutions with values in some $\Omega_\xi$, off a discrete subset, are said to be $S^1$-\emph{invariant}.

\subsection{Extended solutions from the Grassmannian point of view}

Fix on $\mathbbm{C}^{n}$ the standard complex inner product $\langle
\cdot,\cdot\rangle$.
Let $H^n$ be the Hilbert space
 of square-summable $\C^{n}$-valued
functions on the
 circle $S^1$ and $ \langle\cdot,\cdot\rangle_H$ the induced complex inner product. Denote by $H_+^n$ the subspace  formed by those functions of  $H^n$ with Fourier expansion containing only non-negative powers of $\lambda\in S^1$.
 Given a compact Lie group $G\subset U(n)$, we denote by ${Gr}(G)$
 the Grassmannian model \cite{pressley_segal} for the loop group $\Omega G\subset \Omega U(n)$. When $G=U(n)$, we also denote ${Gr}^n:=Gr(U(n))$.

In terms of the Grassmannian model, the  bundle map $u_\xi:U_\xi({G})\to\Omega_\xi$ can be described as follows.
 Take $\gamma\in U_\xi({G})$ and $W=\gamma H^n_+\in {Gr}_{\mathrm{alg}}({G})$, with $\gamma=\sum_{i=s}^{r}\zeta_i\lambda^i$ and $\zeta_r,\zeta_s\neq 0$. Consequently,
 $s$ and $r$ are the smallest integer and the largest integer, respectively, satisfying $\lambda^rH^n_+\subseteq W\subseteq\lambda^{s}H^n_+$.
Fix $\Psi\in \Lambda^+_{\mathrm{alg}}{G}^\C$ such that
$W=\Psi\gamma_\xi H^n_+$. Write
\begin{equation*}
\gamma_\xi H^n_+=\lambda^{s}A^\xi_{s}+\ldots+\lambda^{r-1}A^\xi_{r-1}+\lambda^rH^n_+,
\end{equation*} where the subspaces $A^\xi_i$ define a flag
$\{0\}=A^\xi_{s-1}\subsetneq A^\xi_{s}\subseteq A^\xi_{s+1}\subseteq \ldots\subseteq A^\xi_{r-1}\subsetneq A^\xi_r=\C^n$.
Set $A_i=\Psi(0)A^\xi_i= p_i(W\cap\lambda^iH^n_+),$ where $p_i: H^n \to\C^n$ is defined by
$p_i(\sum\lambda^ja_j)=a_i.$ Then
\begin{equation}\label{popo1}
u_\xi(W)=\lambda^{s}A_{s}+\ldots+\lambda^{r-1}A_{r-1}+\lambda^rH^n_+.
\end{equation}
The loop $\gamma:S^1\to G$ associated to $u_\xi(W)$ under the
identification $\Omega {G} \cong {Gr}({G})$ is the homomorphism
$\gamma(\lambda)=\lambda^s\pi_{s}+\ldots+\lambda^r\pi_{r},$
where $\pi_j$ is the orthogonal projection of $\C^n$ onto the orthogonal complement of $A_{j-1}$ in $A_j$, which we denote by $A_j\ominus A_{j-1}$.

\begin{eg}  \cite{pressley_segal}
  We consider the special orthogonal group  $SO(n)$ as a subgroup of $U(n)$. For each $x$ in $\C^n$ denote
by $\overline{x}$ its complex conjugate. The Grassmannian model of $\Omega SO(n)$ is given by
$${Gr}\left(SO(n)\right)=\big\{{W\in {Gr}^n\,|\,
\overline{W}^{\perp}=\lambda W \big\}}.$$
\end{eg}
\begin{eg}  \cite{pressley_segal}
  If $J:\C^{2n}\rightarrow \C^{2n}$ is the anti-linear map
representing left multiplication by the quaternion $j$, then $X$
in ${U}(2n)$ belongs to ${Sp}(n)$ if, and only if, $XJ=JX$. The
Grassmannian model for $\Omega {Sp}(n)$ is given by
\begin{equation}\label{grsp}
 {Gr}\left(Sp(n)\right)=\{W\in {Gr}^n\,|\,\,
\,JW^{\perp}=\lambda W \}.\end{equation}
\end{eg}

 Take $W:{M} \rightarrow {Gr}({G})$ corresponding to a smooth map $\Phi:M\to \Omega {G}$ under the
identification $\Omega {G} \cong {Gr}({G})$, that is $W=\Phi
H^n_+$.
Segal \cite{segal_1989} has observed that $\Phi$ is an extended solution if, and only if, $W$ satisfies
$\frac{\partial}{\partial z}W  \subseteq  {\lambda}^{-1}W$ (the \emph{pseudo-horizontality} condition) and $\frac{\partial}{\partial \bar{z}} W
 \subseteq W$ ($W$ is a holomorphic vector subbundle of $M\times H^n$) with respect to any local complex coordinate system  $(U,z)$.
If $\Phi:M\setminus D\to U_\xi({G})$ is an extended solution and $W=\Phi H^n_+$, then $u_\xi(W)=u_\xi\circ\Phi H_+^n$ is given pointwise by
(\ref{popo1}) and we get holomorphic subbundles $A_i$ of the trivial bundle $\underline{\C}^n=M\times \C^n$ such that
\begin{equation}\label{flag}
0\subsetneq A_{s} \subseteq \ldots \subseteq A_{r-1} \subsetneq A_r=\underline{\C}^n.
\end{equation}
 The pseudo-horizontally condition  implies that $\frac{\partial}{\partial z}{A_i}\subseteq A_{i+1}$, that is, following the terminology of \cite{burstall_rawnsley_1990}, the flag of holomorphic vector subbundles \eqref{flag} is  \emph{super-horizontal}. Observe that if $A_i=A_{i+1}$, then $A_i$ is both holomorphic and anti-holomorphic, which means that $A_i$ is constant.

Under the identification $\Omega {G}\cong  {Gr}({G})$,
$\mathcal{I}$ induces
 an involution on ${Gr}({G})$, that we shall also denote by $\mathcal{I}$,
 and $\Omega^\mathcal{I} {G}$ can be identified with
$${Gr}^\mathcal{I}({G})=\{W\in {Gr}({G})\,|\,\,
 \,\mbox{if $s(\lambda)\in W$ then $s(-\lambda)\in W$}\}.$$
\begin{lem}\label{lemata} If $W\in{Gr}^\mathcal{I}({G})$ and $\lambda H_+^n\subset W\subset \lambda ^{-1}H_+^n$, then $W$ is $S^1$-invariant, that is it corresponds to some homomorphism  $\gamma:S^1\to G$.
 \end{lem}
\begin{proof} Given $s(\lambda)=s_{-1}\lambda^{-1}+s_0\in W$, since $W$ is fixed by $\mathcal{I}$ we also have  $s(-\lambda)=-s_{-1}\lambda^{-1}+s_0\in W$, which means that $\lambda^{-1}s_{-1}$ and $s_0$ are in $W$.
Hence $W=A_{-1}\lambda^{-1}+A_0+\lambda H_+^n$, with $A_{-1}$ and $A_0$ vector subspaces of $\C^n$.
\end{proof}
\subsection{Normalization of extended solutions}

The following theorem, which is a  generalization of Theorem 4.5 in \cite{burstall_guest_1997}, is fundamental to the classification of extended solutions.

\begin{thm}\label{nor}\cite{correia_pacheco_3} Let $\Phi:M\setminus D\to U_\xi({G})$ be an extended solution. Take $\xi'\in \mathfrak{I}'({G})$ such that $\xi\preceq {\xi'}$ and $\g_0^\xi=\g_0^{\xi'}$. Then there exists some constant loop $\gamma\in \Omega_{\mathrm{alg}}{G}$ such that $\gamma\Phi:M\setminus D\to U_{\xi'}({G})$. \end{thm}

A similar statement holds for extended solutions associated to harmonic maps into inner symmetric spaces.
\begin{thm}\cite{correia_pacheco_4} Let $\Phi:M\setminus D\to U^{\mathcal{I}}_\xi({G})$ be an extended solution. Take $\xi'\in \mathfrak{I}'({G})$ such that $\xi\preceq_\mathcal{I} {\xi'}$. Then there exists some constant loop $\gamma\in \Omega^{\mathcal{I}}_{\mathrm{alg}}{G}$ such that
$\gamma\Phi:M\setminus D\to U^{\mathcal{I}}_{\xi'}({G})$. \end{thm}

Given $\xi=\sum n_iH_i$ and $\xi'=\sum n'_iH_i$ in $\mathfrak{I}'({G})$,  we have $n_i,n'_i\geq 0$ and observe that  $\xi\preceq\xi'$ if and only if $n'_i\leq n_i$ for all $i$.
For each $I\subseteq \{1,\ldots,k\}$, define the cone
$$\mathfrak{C}_{I}=\Big\{\sum_{i=1}^k n_i H_i|\, n_i\geq 0, \,\mbox{$n_j=0$ iff $j\notin I$}\Big\},$$
where $k$ is the number of positive simple roots of $\g$.

\begin{defn} Let $\xi\in\mathfrak{I}'({G})\cap \mathfrak{C}_{I}$. We say that $\xi$ is a \emph{$I$-canonical element} of $G$ with respect to  $\Delta^+$ if it is a maximal element of $(\mathfrak{I}'({G})\cap \mathfrak{C}_{I},\preceq)$, that is,  if $\xi\preceq \xi'$ and $\xi'\in \mathfrak{I}'({G})\cap \mathfrak{C}_{I}$ then $\xi=\xi'$. Similarly, we say that  $\xi$ is a \emph{symmetric canonical element} of $G$ with respect to  $\Delta^+$ if it is a maximal element of $(\mathfrak{I}'({G}),\preceq_{\mathcal {I}})$. \end{defn}

\begin{rem} When $G$ has trivial centre, which is the case considered in \cite{burstall_guest_1997}, the duals $H_1,\ldots,H_k$ belong to the integer lattice. Then, for each $I$ there exists a unique $I$-canonical element, which is given by $\xi_I=\sum_{i\in I}H_i$.
\end{rem}

Any harmonic map $\varphi:M\to G$ with finite uniton number  admits a \emph{normalized extended solution}, that is, an extended solution $\Phi$ taking  values in $U_\xi(G)$ for some canonical element $\xi$. More precisely, we have the following.
\begin{thm}\label{cano}\cite{correia_pacheco_4} Let $\Phi:M \to \Omega_{\rm{alg}}{G}$ be an extended solution. There exist a constant loop $\gamma\in \Omega_{\rm{alg}}{G}$, a subset $I\subseteq \{1,\ldots,k\}$, a $I$-canonical element $\xi'$ and a discrete subset $D\subset M$, such that
$\gamma\Phi(M\setminus D)\subseteq U_{\xi'}({G})$. \end{thm}
\begin{proof}
Take $D\subset M$ and $\xi\in \mathfrak{I}'(\mathrm{G})$ such that $\Phi:M\setminus D \to U_\xi(G)$ . Write $\xi=\sum_{i=1}^kn_iH_i$, with $n_i\geq 0$, and set $I=\{i|\, n_i>0\}$. By Zorn's lemma,  there certainly exists a $I$-canonical element $\xi'$ such that $\xi\preceq \xi'$. On the other hand, from \eqref{gis1} we see that $\g_0^\xi=\g_0^{\xi'}$. Hence the result follows from Theorem \ref{nor}. \end{proof}
\begin{prop} Let ${\rho}:G\to\End(V^{\omega^*})$ be a $n$-dimensional representation of $G=\tilde{G}/C$ with highest weight $\omega^*$ and lowest weight  $\varpi^*$, and $\xi$ a $I$-canonical element of $\g$, where $C$ is a (discrete) subgroup of the  centre of $\tilde{G}$. Then, the  uniton number of  $\Phi:M\setminus D\to U_\xi(G)$ is given by $r_\rho(\Phi)=\omega^*(\xi)-\varpi^*(\xi)$. \end{prop}

\begin{proof} Consider the decomposition of $V^{\omega^*}$ into weight spaces: $V^{\omega^*}=\bigoplus_\omega V_\omega$. The weights $\omega$ have the form $\omega=\omega^*-\sum m_i\alpha_i$, where $m_i\in \mathbbm{N}$ and $\alpha_1,\ldots,\alpha_k$ are the positive simple roots of $\mathfrak{g}$. Since ${\rho}(\exp(X))=\exp(\rho(X))$, we have
$\rho(\gamma_\xi(\lambda))V_\omega=\lambda^{\omega(\xi)}V_\omega.$ On the other hand,  $\alpha_i(\xi)\geq 0$ because  $\xi\in\mathfrak{I}'(G)$.
This means that $\omega(\xi)\leq  \omega^*(\xi)$. Similarly we have   $\omega(\xi)\geq  \varpi^*(\xi)$. Hence $r_\rho(\gamma_\xi)=\omega^*(\xi)-\varpi^*(\xi)$.

It remains to check that $r_\rho(\Phi)=r_\rho(\gamma_\xi)$. Write $\rho(\Phi)=\sum_{i=s}^r\zeta_i\lambda^i$  and, accordingly to the Iwasawa decomposition, $\Phi=\Psi \gamma_\xi \tilde \Psi$, with $\zeta_r,\zeta_s\neq 0$ and $\Psi,\tilde\Psi\in \Lambda^+G^\C$ (and consequently  $\rho(\Psi),\rho(\tilde\Psi)\in \Lambda^+Gl(n,\C)$). We can also write $$\rho(\gamma_\xi)=\sum_{i=\varpi^*(\xi)}^{\omega^*(\xi)} \pi_i \lambda^i.$$
Now, from this it is immediate to see that $s=\varpi^*(\xi)$. On the other hand,
since any representation of a compact Lie group is unitarizable,  $\rho(G)\subset U(n)$ and we have $\rho(\Phi)^{-1}=\rho(\Phi)^*$. Hence we can apply the same argument as before to conclude that $r=\omega^*(\xi)$.
\end{proof}

Henceforth we also denote $r_\rho(\xi):=\omega^*(\xi)-\varpi^*(\xi)$.
Harmonic maps of finite uniton number into  inner symmetric spaces also admit normalized extended solution in the following sense.
\begin{thm}\cite{correia_pacheco_4} Let  $\Phi:M \to \Omega^{\mathcal{I}}_{\rm{alg}}{G}$ be an extended solution with values in $U_{\xi}^{\mathcal{I}}(\mathrm{G})$, for some $\xi\in\mathfrak{I}'(\mathrm{G})$, off a discrete set $D$. There exist a constant loop $\gamma\in \Omega^{\mathcal{I}}_{\rm{alg}}\mathrm{G}$  and a  symmetric canonical element $\xi'$  such that
$\gamma\Phi(M\setminus D)\subseteq U^{\mathcal{I}}_{\xi'}(\mathrm{G})$ and $N_\xi=N_{\xi'}$. \end{thm}

\begin{rem}\label{remark} If $M$ is simply connected and $\varphi:M\to G$ is a harmonic map, then $\varphi$ admits a lift $\tilde{\varphi}:M\to \tilde{G}$, with $\varphi=[\tilde{\varphi}]$, which is also a harmonic map. Let $\tilde{\Phi}$ be a normalized extended solution associated to $\tilde{\varphi}$ with values in $U_\xi(\tilde{G})$, where $\xi$ is a $I$-canonical element of $\tilde{G}$. Projecting onto $G$, this gives an extended solution $\Phi=[\tilde{\Phi}]$ associated to $\varphi$, which is not necessarily a normalized extended solution. From Theorem \ref{cano} and from its proof, we conclude that there exist a constant loop $\gamma$ in $\Omega_{\rm{alg}}{G}$ and a $I$-canonical element $\xi'$ of $G$, with  $\xi\preceq \xi'$, such that $\gamma\Phi:M\setminus D\to U_{\xi'}(G)$.
As we will see in Section \ref{canelm}, in general not all canonical elements $\xi'$ of $G$ satisfy $\xi\preceq \xi'$ for some canonical element $\xi$ of $\tilde G$.
\end{rem}

\subsection{Canonical factorizations of extended solutions} By a \emph{filtration} of length $l$ of an element $\xi \in \mathfrak{I}'(G)$ we simply mean a sequence
$$\xi=\xi_0\precneqq \xi_1\precneqq\xi_2\precneqq\ldots\precneqq\xi_{l-1}\precneqq\xi_l=0$$
of elements $\xi_i$ in $\mathfrak{I}'(G)$.
For example, if $G=\tilde{G}/Z(\tilde G)$ and $I=\{a_0<\ldots<a_{l-1}\}\subseteq \{1,\ldots,k\}$, the sequence defined by
\begin{equation}\label{canfac}
\xi_i=\sum_{j\in I_i}H_j,
\end{equation} with $I_i=\{j\in I |\, j\geq a_i\}$, is a filtration of length $l$ of $\xi=\sum_{j\in I}H_j$.
Given an extended solution $\Phi$ with values in $U_\xi(G)$, a filtration of $\xi$ induces a \emph{canonical factorization} of $\Phi$,
$\Phi=\Phi_{l-1}(\Phi_{l-1}^{-1}\Phi_{l-2})\ldots (\Phi_{1}^{-1}\Phi),$
where, by Theorem \ref{popo}, $\Phi_i=\mathcal{U}_{\xi,\xi_i}(\Phi)$ is also an extended solution.

\section{Canonical elements of classical Lie algebras}\label{canelm}

In \cite{correia_pacheco_4}, we have studied the canonical elements of   $\left(A_n\right)\cong\lsl(n+1,\C)$. This case is hard to describe for a general $n$, but we were able to identify all such canonical elements for $n\leq 4$. Next, we study the $I$-canonical elements for the remaining three classical Lie algebras: $(B_n)\cong\so(2n+1,\C)$, $(C_n )\cong \mathfrak{sp}(n,\C)$ and $\left(D_n\right)\cong \so(2n,\C)$. Observe that  the canonical elements of $\tilde G$ are precisely the integral linear combinations of the vectors $H_i$ which are in $\mathfrak{I}'(\tilde G)\cap \mathfrak{C}_{I}$ (that is, taking account Theorem \ref{lat}, elements which are simultaneously integral linear combinations of the vectors $H_i$ and of the vectors $\eta_i$) and which are maximal with respect to the partial order relation $\preceq$. Hence  $I$-canonical elements shall  be sought among the elements of the finite set formed by the integral combinations $\sum_{i\in I}n_iH_i$, with $n_i\in\{0,\ldots,m_i\}$, which are simultaneously integral linear combinations of the elements $\eta_i$, where
 $m_i$ is the least positive integer which makes $m_iH_i$  an integral combination of the elements $\eta_i$.
Similarly, the symmetric canonical elements shall  be sought among the elements of the finite set formed by the integral combinations $\sum n_iH_i$, with $n_i\in\{0,\ldots,2m_i-1\}$.

We set $E_i=E_{i,i}-E_{n+i,n+i}$, where $E_{j,j}$ is a square matrix whose $(j,j)$ entry is $\sqrt{-1}$ and all other entries are $0$. Consider the algebra of diagonal matrices $\lt=\left\{\sum a_iE_i:\, a_i\in \mathbbm{R}\right\}.$ The complexification $\lt^\C$  is a Cartan subalgebra whether the Lie algebra $\g^\C$ is $(B_n)$, $(C_n )$ or $\left(D_n\right)$. Let $L_1,\ldots,L_n$ be the dual basis in $\sqrt{-1}\lt^*$ of $E_1,\ldots,E_n$:  $L_i(E_j)=\sqrt{-1}\delta_{ij}$. For details on the structure of these Lie algebras and their representations see \cite{fulton_harris}.

\subsection{Canonical elements of $(B_n)$}

The roots are the vectors $\pm L_i\pm L_j$, with $i\neq j$ and $1\leq i,j\leq n$, together with $\pm L_i$. Fix the  positive root system
$\Delta^+=  \{L_i\pm L_j\}_{i<j}\cup\{L_i\}_i.$ The positive simple roots are then the roots $\alpha_i=L_i-L_{i+1}$, for $1\leq i\leq n-1$, and $\alpha_n=L_n$. The
   fundamental weights are $\omega_i=L_1+L_2+\ldots+L_i$, for $1\leq i\leq n-1$, and $\omega_n=\left(L_1+\ldots+L_n\right)/2.$
  The dual basis $H_1,\ldots,H_n\in \lt$ of  $\alpha_1,\ldots,\alpha_n\in \sqrt{-1}\lt^*$ is given by
\begin{align*} H_i=E_1+E_2+\ldots+E_i=\eta_1+2\eta_2+\ldots+i\eta_i+\ldots+i\eta_{n-1}+\frac{i}{2}\eta_n, \end{align*} where $\eta_1,\ldots,\eta_n\in \lt$ is the dual basis of $\omega_1,\ldots,\omega_n\in\sqrt{-1}\lt^*$.
The simply connected Lie group with Lie algebra $\g=\so(2n+1)$ is $\tilde G=Spin(2n+1)$ and its centre is $\Z_2$. The standard representation has highest weight $\omega_1$ and the adjoint representation has highest weight $\omega_2$.
Denote by $\rho_k$ the irreducible representation of $\tilde{G}$ with highest weight $\omega_k$.

\begin{thm}\label{oddspin}
 Fix $I\subseteq \{1,\ldots,n\}$ and set $\xi_I=\sum_{i\in I}H_i$.
If $\sum_{i\in I}i$ is even, $\xi_I$ is the unique $I$-canonical element of $Spin(2n+1)$, and
    \begin{equation}\label{unir}    r_{\rho_k}(\xi_I)=2|I|+2|I\cap\{2,\ldots,n\}|+\ldots +2|I\cap\{k,\ldots,n\}|, \, \mbox{for $k<n$}.\end{equation}
    If $\sum_{i\in I}i$ is odd, the $I$-canonical elements of $Spin(2n+1)$  are precisely the elements of the form  $\xi^j_{I}=H_j+\xi_I$, with $j\in I$  odd, and
    $r_{\rho_k}(\xi_{I}^j)= r_{\rho_k}(\xi_I)+2\min\{j,k\},\, \mbox{for $k<n$}.$
    When $k=n$ the uniton numbers  $r_{\rho_n}(\xi_I)$ and $r_{\rho_n}(\xi_{I}^j)$ are given by the above formulas divided by $2$.
\end{thm}
\begin{proof}
  We only have to observe that
  \begin{align*}
  \xi_I=\sum_{i\in I} H_i=|I|E_1+|I\cap\{2,\ldots,n\}|E_2+\ldots +|I\cap\{n-1,n\}|E_{n-1}+|I\cap\{n\}|E_{n},
  \end{align*}
  and that  the $\eta_j$-coefficient of  $\sum_{i\in I} H_i$ is integer for $j\leq n-1$ whereas the $\eta_n$-coefficient is $\sum_{i\in I}i/2$.
  If $\sum_{i\in I}i$ is even, the $\eta_n$-coefficient is also integer and, consequently, $\xi_I$ is the unique $I$-canonical element. If   $\sum_{i\in I}i$ is odd the  $\eta_n$-coefficient is
  half-integer. Since an element $H_j$, with $j\in I$ odd, also has half-integer $\eta_n$-coefficient, $H_j+\xi_I$ is $I$-canonical. Clearly, in this case, all $I$-canonical elements must be of this form.

  Since the lowest weight of $\rho_k$ is $-\omega_k$, we have  $r_{\rho_k}(\xi)=2\omega_k(\xi)$ for each $\xi\in \mathfrak{I}'(G)$.
\end{proof}

Since $SO(2n+1)\cong Spin(2n+1)/\mathbbm{Z}_2$ has trivial centre,  for each $I\subseteq \{1,\ldots,n\}$, the unique $I$-canonical element is given by $\xi_I=\sum_{i\in I}H_i$.
By Theorem \ref{repG}, the representation of $Spin(2n+1)$ with highest weight $\omega^*=\sum \lambda_iL_i$ is a representation of $SO(2n+1)$ if and only if the numbers $\lambda_j$ are all integers.
 In particular, each $\omega_k$, with $k\leq n-1$, is the highest weight of an irreducible representation of $SO(2n+1)$ and
 a normalized extended solution $\Phi$ with values in $U_{\xi_I}(SO(2n+1))$ has uniton number $r_{\rho_k}(\Phi)$ given by $\eqref{unir}$. In particular, for the standard representation we have $r_{\rho_1}(\Phi)\leq 2n$, whereas for the adjoint representation we have $r_{\rho_2}(\Phi)\leq 4n-2$, and the equalities hold if and only if $I=\{1,\ldots,n\}$. This agrees with the minimal uniton number estimations in \cite{svensson_wood_2010}   for the standard representation and in \cite{burstall_guest_1997} for the adjoint representation.
\subsubsection{Harmonic maps into $Spin(2n+1)$ and the spinor representation.}\label{spinorharmonic}
We start by recalling the construction of the spinor representation of $Spin(2n+1)$ by means of Clifford and exterior algebras. These are representations which do not descend to representations of $SO(m)$.  The corresponding highest weight is precisely $\omega_n=\left(L_1+\ldots+L_n\right)/2.$ We refer the reader to the Lecture 20 of \cite{fulton_harris} for details.

Take $V=\mathbbm{C}^{2n+1}\cong (\mathbbm{R}^{2n+1})^\mathbbm{C}$ and let  $Q$ be the complex-bilinear extension of the standard inner product on $\mathbbm{R}^{2n+1}$ to $V$.
The corresponding Clifford algebra $Cl(Q)$  is an associative algebra containing and generated by $V$, with unit $1$ and $v\cdot v=Q(v,v)1.$ We also have a structure of Lie algebra on $Cl(Q)$, with bracket $[a,b]=a\cdot b-b\cdot a$. Denote by $Cl^{+}(Q)$ (respectively, $Cl^{-}(Q)$) the subalgebra spanned by products of an even number (respectively, odd number) of elements in $V$.
The Lie algebra $\mathfrak{so}(Q)$ can be embedded in  $Cl^{+}(Q)$ as follows. Identify $\wedge^2 V$ with $\mathfrak{so}(Q)\subset \mathrm{End}(V)$  by using the isomorphism $a\wedge b\mapsto \varphi_{a\wedge b}$ defined by
$$\varphi_{a\wedge b}(v)=2(Q(b,v)a-Q(a,v)b).$$ The mapping $\rho:\wedge^2 V\cong \mathfrak{so}(Q)\to Cl^{+}(Q)$ defined by
$$\rho(a\wedge b)=\frac12[a,b]$$
embeds  $\mathfrak{so}(Q)$ as a Lie subalgebra of $Cl^{+}(Q)$.

Now, write $V=U\oplus \overline{U}\oplus R$, where $U$ is a $n$-dimensional isotropic subspace and $R$ is one-dimensional and perpendicular to both $U$ and $\overline{U}$. This decomposition determines an isomorphism of algebras
$$L\oplus L':Cl(Q)\to \mathrm{End}(\wedge ^\bullet U )\oplus\mathrm{End}(\wedge ^\bullet \overline{U}),$$
where $\wedge^\bullet U=\wedge^0U\oplus\wedge^1U\oplus\ldots\oplus \wedge^n U$. By the universal property, to define such isomorphism it is suffices to establish how it restricts to $V\subset Cl(Q)$. If $w\in U$, $L(w)$ is the left multiplication by $w$ on $\wedge^\bullet U$. If $w'\in\overline{U}$, consider its dual $\vartheta\in U^*$ with respect to $2Q$, that is $\vartheta(w)=2Q(w,w')$ for all $w\in U$. Then $L(w')$ is the derivation on $\wedge^\bullet U$ such that $L(w')(1)=0$, $L(w')(w)=\vartheta(w)\in \wedge^0U=\mathbbm{C}$ for $w\in U=\wedge^1U$, and
$$L(w')(\zeta\wedge \xi)=L(w')(\zeta)\wedge \xi+(-1)^{\mathrm{deg}\zeta }\zeta\wedge L(w')(\xi).$$ Given $e_0\in R$ with $Q(e_0,e_0)=1$,  $L(e_0)$ is the identity on $\wedge^{\mathrm{even}}U$ and minus the identity on $\wedge^{\mathrm{odd}}U$. The map $L'$ to $\mathrm{End}(\wedge ^\bullet \overline{U} )$ is defined similarly by reversing the roles of $U$ and $\overline{U}$.

The subalgebra $Cl^+(Q)$ of $Cl(Q)$ is mapped by $L$ isomorphically  to $\mathrm{End}(\wedge ^\bullet U )$.
Hence, $S=\wedge ^\bullet U$ is a representation, the \emph{spinor representation}, of the Lie algebra $\mathfrak{so}(Q)$:
\begin{equation}\label{spinor}
\rho:\mathfrak{so}(Q)\cong \wedge^2 V\to Cl^+(Q) \cong \mathrm{End}(\wedge ^\bullet U).
\end{equation}
It turns out that this representation is irreducible with  highest weight $\omega_n$. Tracing through the identifications in \eqref{spinor}, we can compute explicitly the spinor action for the elements $E_i=E_{i,i}-E_{n+i,n+i}\in \mathfrak{t}\subset \mathfrak{so}(Q)$. Take a  basis $e_1,\ldots, e_n,\overline{e}_1,\ldots,\overline{e}_n,e_{0}$ for $V$, where $e_1,\ldots, e_n$ span $U$ and $e_0$ spans $R$, satisfying $Q(e_i,\overline{e}_i)=Q(e_0,{e}_0)=1$ and $Q(e_i,\overline{e}_j)=0$ for all other pairs $i,j$. Then $E_i(e_i)=\sqrt{-1}e_i$ and $E_i(\overline{e}_i)=-\sqrt{-1}\overline{e}_i$, which means that
$E_i$ corresponds to $\frac{\sqrt{-1}}{2}(e_i\wedge \overline{e}_i)\in\wedge^2 V$. This in turn corresponds to $\frac{\sqrt{-1}}{2}(e_i\cdot \overline{e}_i-1)\in Cl^+(Q)$. Taking into account the definition of $L$, we conclude, considering the basis vectors $e_I=e_{i_1}\wedge\ldots \wedge e_{i_r}$ for $\wedge^\bullet U$, that
\begin{equation}\label{e_I}
\rho(E_i)(e_I)=\left\{\begin{array}{cl}
            \frac{\sqrt{-1}}{2}e_I & \mbox{if $i\in I$}  \\
           -\frac{\sqrt{-1}}{2}e_I & \mbox{if $i\notin I$}
          \end{array}\right.,
\end{equation}
with $I=\{i_1,\ldots,i_r\}\subseteq \{1,\ldots,n\}$.

Let $P:\wedge ^\bullet U\to\wedge ^n U$ be the projection onto $\wedge ^n U $ with respect to the decomposition $\wedge^\bullet U=\wedge^0U\oplus\wedge^1U\oplus\ldots\oplus \wedge^n U$ and fix an isomorphism $\wedge ^n U\cong \mathbbm{C}$. Define the nondegenerated bilinear form $\beta$ on $\wedge ^\bullet U$ as follows. Given $u,v\in \wedge ^\bullet U$, set $\beta(u,v)=P(\tau(u)\wedge v)$, where $\tau$ is the \emph{reversing antiautomorphism}
$\tau(v_{i_1}\wedge \ldots \wedge v_{i_r})= v_{i_r}\wedge \ldots \wedge v_{i_1}$
 if $n$ is even,  and the \emph{conjugation antiautormphism}
 $\tau(v_{i_1}\wedge \ldots \wedge v_{i_r})=(-1)^r v_{i_r}\wedge \ldots \wedge v_{i_1}$
 if $n$ is odd.  The bilinear form $\beta$ is symmetric if $n$ is congruent to $0$ or $3$ modulo 4, and skew-symmetric otherwise. Moreover, $\beta$ is preserved by $Spin(2n+1)$. So the spin representation $\rho$ gives an homomorphism
 $$\begin{array}{ll}
     Spin(2n+1)\to SO(2^n) & \mbox{if $n=0,3 \mod 4$} \\
     Spin(2n+1)\to Sp(2^{n-1}) & \mbox{if $n=1,2 \mod 4$},
   \end{array}$$
which is an isomorphism in the lowest dimensional cases:
$Spin(3)\cong Sp(2)\cong  SU(2)$ and $Spin(5)\cong Sp(2).$ According to Theorem \ref{oddspin},
 $2H_1$ is the unique non-trivial canonical element for $Spin(3)\cong SU(2)$.  The corresponding extended solutions in $\Omega Spin(3)\cong \Omega SU(2)$ are described in Section 5.3 of \cite{correia_pacheco_4}  with respect to the standard representation of $SU(2)$. For $Spin(5)\cong Sp(2)$, we have three non-trivial canonical elements: $H_1$, $2H_2$ and $2H_2+H_1$. We will study this case with more detail in Section \ref{cn}.

Next we consider the case $n=3$.
 Define on $\mathbbm{R}^8\cong \mathbbm{O}$ the \emph{triple cross product}:
$$x\times y\times z=-\frac12(x(\bar yz-z(\bar yx)),$$
where $\mathbbm{O}$ is the division algebra of octonions. The simple connected Lie group $Spin(7)$, seen  as a subgroup of $SO(8)$ via spinor representation, is given, up to conjugation, by
$$ Spin(7)=\{g\in SO(8)|\, gu\times gv\times gw=g(u\times v\times w) \forall u,v,w\in  \mathbbm{O}\}$$
 (see Chapter IV of \cite{harvey_lawson}  for details).
Hence, the Grassmannian model of $\Omega Spin(7)$ with respect to the spinor representation  is given by the following proposition, whose proof we omit since it is completely similar to that of Proposition 3.2 in \cite{correia_pacheco_3}.
\begin{prop}\label{grassmodelspin}
A subspace $W\in {Gr}\big(SO(8)\big)$ corresponds to
a loop in $Spin(7)$ if, and only if, it belongs to
$${Gr}(Spin(7))=\{W\in {Gr}\big(SO(8)\big)\,|\,
\,W^{sm}\times W^{sm}\times W^{sm}\subseteq W^{sm}\}$$ (here $W^{sm}$ denotes the
subspace  of smooth functions in $W$).
\end{prop}

Fix on $\wedge^\bullet U$, where $U=\mathrm{span}\{e_1,e_2,e_3\}$, the vector basis
$$v_1=1,\,v_2=e_1,\,v_3=e_2,\,v_4=e_3,\,\overline{v}_1=e_1\wedge e_2\wedge e_3,\,\overline{v}_2=e_2\wedge e_3,\,\overline{v}_3= e_3\wedge e_1,\,\overline{v}_4= e_1\wedge e_2.$$
Consider the nondegenerated bilinear form $\beta$ on $\wedge^\bullet U$ defined as above, where we use $e_1\wedge e_2\wedge e_3$ to identify $\wedge^3U$ with $\mathbbm{C}$. Hence the 4-dimensional space spanned by $v_1,v_2,v_3,v_4$ is isotropic, $\beta(v_i,\overline{v}_i)=1$ and $\beta(v_i,\overline{v}_j)=0$ for the other pairs  $i,j$. We can now use  \eqref{e_I} and Proposition \ref{grassmodelspin} to describe the $S^1$-invariant extended solutions with values in $\Omega Spin(7)$ associated to the canonical elements of $Spin(7)$, which are shown in Table \ref{cococo1}. We illustrate this with some examples.
    \begin{table}[!htb]
\begin{tabular}{c|c|c|c} $Spin(7)$ & $|I|=3$ & $|I|=2$ & $|I|=1$ \\ \hline  &$H_1+H_2+H_3$ & $2H_1+H_2$ & $2H_1$  \\ &   & $H_1+H_3$ & $H_2$ \\ &   & $H_2+2H_3$ & $2H_3$

\end{tabular}
 \vspace{.10in}
\caption{Canonical elements for $Spin(7)$.}\label{cococo1}
\end{table}

Consider first the canonical element $\xi=2H_1$, for which $r_{\rho_n}(\xi)=2$. We have
$$\rho_n(\gamma_\xi) H_+^8=\lambda^{-1}A_{-1}^\xi+A^\xi_{-1}+\lambda H_+^8,$$ where $A^\xi_{-1}$ is the 4-dimensional isotropic vector space spanned by $v_1,v_3,v_4,\overline{v}_2$. By Proposition \ref{grassmodelspin}, this vector space satisfies $A^\xi_{-1}\times A^\xi_{-1}\times A^\xi_{-1}=0$ and $A^\xi_{-1}\times A^\xi_{-1}\times \overline{A^\xi}_{-1}\subseteq A^\xi_{-1}$. Hence the $S^1$-invariant extended solutions associated to $\xi=2H_1$ are of the form $W=\lambda^{-1}A+ A+ \lambda H_+^8$, where $A$ is a rank 4, isotropic and constant
 vector subbundle of $\underline{\C}^8$ satisfying $A\times A\times A=0$ and $A\times A\times \overline{A}\subseteq A$.

The canonical element $\xi=H_2$ has also $r_{\rho_n}(\xi)=2$. We have
$$\rho_n(\gamma_\xi) H_+^8=\lambda^{-1}A_{-1}^\xi+\overline{A^\xi_{-1}}^\perp+\lambda H_+^8,$$ where $A_{-1}^\xi$ is the 2-dimensional isotropic vector space spanned by $v_1,v_4$. Hence the $S^1$-invariant extended solutions associated to $\xi=H_2$ are of the form $W=\lambda^{-1}A+ \overline{A}^\perp+\lambda H_+^8,$ where $A$ is a holomorphic subbundle of  $\underline{\C}^8$,
with $\frac{\partial}{\partial z}A\subseteq \overline{A}^\perp$, which satisfies: \emph{(i)} $A$ is isotropic and has rank 2; \emph{(ii)} 
by Proposition \ref{grassmodelspin},
$$A\times A\times \overline{A}^\perp=0,\, A\times A\times \mathbbm{O}\subseteq A,\,A\times \overline{A}^\perp\times \overline{A}^\perp\subseteq A,\,\overline{A}^\perp\times \overline{A}^\perp\times\overline{A}^\perp\subseteq \overline{A}^\perp,\, A\times \overline{A}^\perp\times \mathbbm{O}\subseteq \overline{A}^\perp. $$

The canonical element with highest uniton number is $\xi=H_2+2H_3$, for which $r_{\rho_n}(\xi)=8$.  In this case
the corresponding $S^1$-invariant extended solutions are of the form
$$W=\lambda^{-4}A+ \lambda^{-3}A+\lambda^{-2}B+\lambda^{-1}C + {C}+ \lambda\overline{B}^\perp+\lambda^2\overline{A}^\perp+  \lambda^3 H_+^8.$$
Additionally to the conditions that follow from Proposition \ref{grassmodelspin}, we have:  $A$ is constant and has rank 1, $B$ has rank 2, and $C$ is constant and has rank 4.

\subsubsection{The canonical factorization.}
Next we analyse  the canonical factorization  of extended solutions $\Phi$ with values in $U_{\xi}(SO(2n+1))$ associated to the filtration defined by  \eqref{canfac}, with $I=\{1,\ldots,n\}$.
 The matrices $E_i$ are written in terms of a basis $e_1,\ldots,e_{2n+1}$ of $\C^{2n+1}$ with $\langle e_i,e_{i+n}\rangle =1$, $\langle e_{2n+1},e_{2n+1}\rangle =1$ and $\langle e_i,e_{j}\rangle =0$ if $j\neq i+n$.
We have
$$\xi=H_1+\ldots+H_n=nE_1+(n-1)E_2+\ldots+ 2E_{n-1}+E_n.$$  The element of ${Gr}(SO(2n+1))$ corresponding to $\gamma_\xi(\lambda)$  is given by
\begin{equation}\label{wxi}
W_{\xi}=\gamma_\xi H_+^{2n+1}=\lambda^{-n}A^\xi_{-n}+\ldots+\lambda^{-1}A^\xi_{-1}+\overline{A^\xi_{-1}}^\perp +\ldots +\lambda^{n-1}\overline{A^\xi_{-n}}^\perp+\lambda^nH^{2n+1}_+,
\end{equation}
where $\{0\}\neq A^\xi_{-n}\subsetneq A^\xi_{-n+1}\subsetneq \ldots\subsetneq A^\xi_{-1}$
is a flag of isotropic subspaces of $\C^{2n+1}$.  On the other hand, the element of ${Gr}(SO(2n+1))$ corresponding to $\gamma_{\xi_1}(\lambda)$, with
$$\xi_1=H_2+\ldots+H_n=(n-1)(E_1+E_2)+\ldots+ 2E_{n-1}+E_n,$$
  is given by
\begin{equation}\label{wxi1}
W_{\xi_1}=\gamma_{\xi_1} H_+^{2n+1}=\lambda^{-n+1}A^\xi_{-n+1}+\ldots+\lambda^{-1}A^\xi_{-1}+\overline{A^\xi_{-1}}^\perp +\ldots +\lambda^{n-2}\overline{A^\xi_{-n+1}}^\perp+\lambda^{n-1}H^{2n+1}_+.
\end{equation}
From \eqref{wxi} and \eqref{wxi1} we see that
\begin{equation}\label{wxis}
 W_{\xi_1}=W_{\xi}\cap \lambda^{-n+1}H_+^{2n+1}+\lambda^{n-1} H_+^{2n+1}.
\end{equation}
Now, take a loop $\Phi\in U_\xi(SO(2n+1))$ and set $W=\Phi H_+^{2n+1}$. By definition of  $U_\xi(SO(2n+1))$, we can write $W=\Psi \gamma_\xi H^{2n+1}_+$ for some $\Psi\in \Lambda_+ SO(2n+1,\mathbbm{C})$. From \eqref{wxis} we get
\begin{align*}
 W^1&:=\Psi \gamma_{\xi_1} H^{2n+1}_+ =\Psi W_{\xi_1}=\Psi \{W_{\xi}\cap \lambda^{-n+1}H_+^{2n+1}+\lambda^{n-1} H_+^{2n+1}\} \\&= \Psi \gamma_{\xi}H_+^{2n+1}\cap \lambda^{-n+1}\Psi H_+^{2n+1}+\lambda^{n-1} \Psi H_+^{2n+1}= W\cap \lambda^{-n+1}H_+^{2n+1}+\lambda^{n-1} H_+^{2n+1}.
\end{align*}
 More generally, for $\Phi_i=\mathcal{U}_{\xi,\xi_i}(\Phi)$ and $W^i:=\Phi_i H_+^{2n+1}=\Psi \gamma_{\xi_i}H_+^{2n+1}$, with $i=1,\ldots,n$, we can check that
$$W^i=W\cap \lambda^{-n+i}H_+^{2n+1}+\lambda^{n-i} H_+^{2n+1}.$$
Hence we  conclude that the canonical factorization associated to the filtration defined by  \eqref{canfac}
 is precisely the factorization defined on Section 6 of \cite{svensson_wood_2010}.

\subsubsection{Symmetric canonical elements of $(B_n)$.}
Let us describe the symmetric canonical elements of ${Spin}(2n+1)$. As we have observed above, the symmetric canonical elements of ${Spin}(2n+1)$ shall be sought among the elements of the finite set formed by the integral linear combinations $\sum_{i=1}n_iH_i$, with $n_i\in\{0,1,2,3\}$ if $i$ is odd and $n_i\in\{0,1\}$ if $i$ is even, which are simultaneously integral linear combinations of the elements $\eta_i$. Hence we have the following.
\begin{thm}
Take $I\subseteq \{1,\ldots,n\}$ and set  $\xi_I=\sum_{i\in I}H_i$.  If $\sum_{i\in I} i$ even, the elements $\xi_I$ and $\xi_I+2H_j$, for each $j\in \{1,\ldots,n\}$ odd, are symmetric canonical elements of  ${Spin}(2n+1)$.
All the symmetric canonical elements of ${Spin}(2n+1)$ are  obtained in this way.
\end{thm}

\subsubsection{Real Grassmannians.}

Given $\xi\in \mathfrak{I}'(SO(2n+1))$, set
\begin{equation}\label{m+m-}
m_\xi^+=\sum_{\mbox{$i$ \tiny{odd}}}\dim A^\xi_i\ominus A^\xi_{i-1}, \quad m_\xi^-=\sum_{\mbox{$i$ \tiny{even}}}\dim A^\xi_i\ominus A^\xi_{i-1},
  \end{equation}
  where $A^\xi_i\ominus A^\xi_{i-1}$ denotes the orthogonal complement of   $A^\xi_{i-1}$ in $A^\xi_i$. The number $m_\xi^+$ is necessarily even and $m_\xi^-$ is necessarily odd.
For each $m\leq 2n+1$, let $G^{\R}_m(2n+1)$ be the \emph{real Grassmannian}  of complex $m$-dimensional subspaces $V$ of $\C^{2n+1}$ satisfying $V=\overline V$.
We identify the inner symmetric space $N_\xi=\{g\gamma_\xi(-1)g^{-1}\}$  with
$G^{\R}_{m}(2n+1)$, where $m=\min \{m_\xi^+,m_\xi^-\}$, via the totally geodesic  embedding
$\iota:G^{\R}_{m}(2n+1)\hookrightarrow \sqrt{e}\subset SO(2n+1)$ defined by $\iota(V)=\pi_V^\perp-\pi_V$ if $m$ is even (if $m$ is odd,  $\pi_V^\perp-\pi_V$ is not in $SO(2n+1)$) and $\iota(V)=\pi_V-\pi_V^\perp$ if $m$ is odd, where $\pi_V$ denotes the orthogonal projection  onto $V$.

\begin{prop}
  Given a symmetric canonical element $\xi$ of $SO(2n+1)$ and $m\leq n$, if $N_\xi=\iota( G^{\R}_m(2n+1))$ then
  $r_{\rho_k}(\xi)\leq 2mk$ if $k\leq n-m+1$, and $$r_{\rho_k}(\xi)\leq 2m(n-m+1)+(k-n+m-1)(m+n-k)$$ if $n>k> n-m+1$. In particular, for the standard representation we have  $r_{\rho_1}(\xi)\leq 2m$ and for the adjoint representation we have $r_{\rho_2}(\xi)\leq 4m$ if $m<n$ and  $r_{\rho_2}(\xi)\leq 4m-2$ if $m=n$.

\end{prop}
\begin{proof}
  We only have to observe that, whether $m$ is even or odd, the symmetric canonical element  associated to $G^{\R}_m(2n+1)$ that maximizes the uniton number $r_{\rho_k}(\xi)$ is the element $\xi_I=\sum_{i\in I}H_i$ with  $I=\{n-m+1, n-m+2,\ldots, n\}$.
\end{proof}
This  agrees with the estimations of the minimal uniton number of harmonic maps into $G^{\R}_m(2n+1)$ in \cite{burstall_guest_1997} (for the adjoint representation) and \cite{svensson_wood_2010} (for the standard representation).

Taking $m=1$, the element $\xi=H_n$ is the unique symmetric canonical element associated to the real projective space $\mathbbm{R}P^{2n}$, and we have $r_{\rho_1}(\xi)= 2$.
Let $\Phi:M\setminus{D}\to U^{\mathcal{I}}_\xi(SO(2n+1))$ be an extended solution. The corresponding map $W=\Phi H^{2n+1}_+$ satisfies $\lambda H^{2n+1}_+\subset W \subset \lambda^{-1} H^{2n+1}_+$.  By Lemma \ref{lemata}, this means  that $\Phi$ must be $S^1$-invariant. Hence
 all harmonic maps of finite uniton number in $\mathbbm{R}P^{2n}$ have the form
$$\varphi=\overline{A}_{-1}^\perp\ominus A_{-1},$$
with $A_{-1}$ an isotropic holomorphic subbundle of $M\times \mathbbm{C}^{2n+1}$ with rank $n$. This is essentially the classification given in Corollary 6.11 of \cite{eells_wood_1983}. In fact, if $\varphi$ is full, $A_{-1}$ is the $n$-\emph{osculating space} of some full totally isotropic holomorphic map $f:M\to \mathbbm{C}P^{2n}$, the so called \emph{directrix curve}  of $\varphi$. That is, in a local system of coordinates $(U,z)$, we have
$$A_{-1}(z)=\mathrm{Span}\big\{g,\frac{\partial g}{\partial z},\frac{\partial^2 g}{\partial z^2},\ldots, \frac{\partial^n g}{\partial z^n}\big\},$$
where $g$ is a lift of $f$ over $U$.

The symmetric canonical elements associated to $G^{\R}_2(2n+1)$ are the elements $\xi_1=H_1$, with $r_{\rho_1}(\xi_1)=2$, and $\xi_i=H_i+H_{i-1}$, with $r_{\rho_1}(\xi_i)= 4$ and $i\in\{2,\ldots,n\}$. Harmonic maps into $G^{\R}_2(2n+1)$  corresponding $\xi_1=H_1$ are of the form $\varphi=f\oplus \bar{f}$, with $f:M\to \mathbbm{R}P^{2n+1}$ holomorphic, that is, $\varphi$ is a \emph{real mixed pair}
\cite{Bahy_Wood_1989}. The harmonic maps associated to the canonical elements $\xi_i=H_i+H_{i-1}$ have a characterization similar to that  of harmonic maps of finite uniton number into the quaternionic projective space $\mathbbm{H}P^{n-1}$ that we will discuss in detail later. For example, the $S^1$-invariant extended solutions associated to the  canonical elements $\xi_i$, with $i=2,\ldots, n$,  are of the form
$$W=A_{-2}\lambda^{-2}+A_{-1}\lambda^{-1}+\overline{A}_{-1}^\perp+\overline{A}_{-2}^\perp+\lambda^2 H_+^{2n},$$ with
 $A_{-2}$ and $A_{-1}$ holomorphic isotropic subbundles of ranks $i-1$ and $i$, respectively. These extended solutions produce harmonic maps of the form $\varphi=h\oplus \bar h$, with $h= A_{-1}\ominus A_{-2}$.

\subsection{Canonical elements of $(C_n)$.}\label{cn}

The roots are the vectors $\pm L_i\pm L_j$, with  $1\leq i,j\leq n$.
Fix the  positive root system
$\Delta^+=\{L_i+L_j\}_{i\leq j}\cup\{L_i-L_j\}_{i<j}.$
The positive simple roots are then the roots $\alpha_i=L_i-L_{i+1}$,  for $1\leq i\leq n-1$, and $\alpha_n=2L_n.$ The fundamental weights are
$\omega_i=L_1+\ldots+L_i$, for all $1\leq i\leq n$.
  The dual basis $H_1,\ldots,H_n\in \lt$ of  $\alpha_1,\ldots,\alpha_n\in \sqrt{-1}\lt^*$ is given by
\begin{align*} H_i & =E_1+E_2+\ldots+E_i=\eta_1+2\eta_2+\ldots+i\eta_i+\ldots+i\eta_n\quad\mbox{for $i\leq n-1$}\\ H_n & =\frac12\left(E_1+E_2+\ldots+E_n\right)=\frac12\left(\eta_1+2\eta_2+\ldots+n\eta_n\right)\!,\end{align*} where $\eta_1,\ldots,\eta_n\in \lt$ is the dual basis of $\omega_1,\ldots,\omega_n\in\sqrt{-1}\lt^*$.
The simply connected Lie group with Lie algebra $\g= \mathfrak{sp} (n)$ is $\tilde G=Sp(n)$ and its centre is $\Z_2$. The standard representation has highest weight $\omega_1$ and the adjoint representation has highest weight $2\omega_1$.

As before, let $m_i$ be the least positive integer which makes $m_iH_i$  an integral linear combination of the elements $\eta_i$.  Then $m_i=1$ if $1\leq i\leq n-1$, and $m_n=2$. Taking account the general procedure, we have the following.
\begin{thm}
  Given $I\subseteq \{1,\ldots,n\}$, set $\xi_I=\sum_{i\in I}H_i$. The element $\xi_I+\delta_I(n)H_n$ is the unique $I$-canonical element of $Sp(n)$, where $\delta_I(j)=1$ if $j\in I$ and $\delta_I(j)=0$ if $j\notin I$.
  Moreover,   $r_{\rho_k}(\xi_I+\delta_I(n)H_n)$ is given by \eqref{unir} for all $k\leq n$.
   \end{thm}

In particular, for the standard representation, we see that the corresponding normalized extended solutions $\Phi$ have uniton number $r_{\rho_1}(\Phi)\leq 2n$, which agrees with Theorem 2 of \cite{pacheco_2006}, where the equality holds if and only if $I=\{1,\ldots,n\}$.

Let us consider the case $n=2$. We have three canonical elements  for $Sp(2)$: $\xi_1=H_1$, $\xi_2=2H_2$, and $\xi_3=H_1+2H_2$. The corresponding minimal uniton numbers with respect to the standard representation are given by $\rho_1(\xi_1)=\rho_1(\xi_2)=2$ and $\rho_1(\xi_3)=4$. Taking account \eqref{grsp} and the harmonicity conditions, we see that the extended solutions $W:M\to Gr(Sp(2))$ associated to $\xi_i$, with $i=1,2$, have the following form:
$$W=R_\lambda+JA^\perp +\lambda H^4_+$$ where $A$ is a $J$-isotropic holomorphic rank $i$ subbundle of $M\times \C^4$ satisfying $\frac{\partial}{\partial z}A\subset JA^\perp$; $R_\lambda$ is a holomorphic subbundle of $\lambda^{-1}A+JA$ such that \emph{(i)} $p_{-1}(R_\lambda)=A$ almost everywhere and \emph{(ii)}  $R_1$ is a $J$-isotropic subbundle of $M\times \C^4$.

\subsubsection{The canonical factorization.}
Consider the filtration $\xi=\xi_0\precneqq \xi_1\precneqq\ldots\precneqq\xi_{n-1}\precneqq\xi_n=0$, where $\xi_i=H_{i+1}+\ldots+H_{n-1}+2H_n$ if $i\leq n-2$ and  $\xi_{n-1}=2H_n$. The matrices $E_i$ are written in terms of a complex orthonormal basis $e_1,\ldots,e_{2n}$ of $\C^{2n}$ with $e_{i+n}=Je_i$, where  $J:\C^{2n}\rightarrow \C^{2n}$ is the anti-linear map
representing left multiplication by the quaternion $j$.
 The element of ${Gr}(Sp(n))$ corresponding to $\gamma_\xi(\lambda)$  is given by
$$W_{\xi}=\gamma_\xi H_+^{2n+1}=\lambda^{-n}A^\xi_{-n}+\ldots+\lambda^{-1}A^\xi_{-1}+J{A^\xi_{-1}}\!\!^\perp +\ldots +\lambda^{n-1}J{A^\xi_{-n}}\!\!^\perp+\lambda^nH^{2n}_+,$$
where $\{0\}\neq A^\xi_{-n}\subsetneq A^\xi_{-n+1}\subsetneq \ldots\subsetneq A^\xi_{-1}$
is a flag of $J$-isotropic subspaces of $\C^{2n}$.  As above, one can check that $W^i=\Psi \gamma_{\xi_i}H_+^{2n}$, with $i=1,\ldots,n$, satisfies
\begin{equation}\label{wi}
W^i=W\cap \lambda^{-n+i}H_+^{2n}+\lambda^{n-i} H_+^{2n}.
\end{equation}
This is precisely the factorization defined in Section 3.4 of \cite{pacheco_2006}.

\subsubsection{Symmetric canonical elements of $(C_n)$.}
Let us start by establish  the symmetric canonical elements of ${Sp}(n)$.
\begin{thm}
The symmetric canonical elements of ${Sp}(n)$ are precisely the elements of the form $\xi_I+\delta_I(n)H_n$, with  $I\subseteq \{1,\ldots,n\}$.
\end{thm}

Given $\xi\in \mathfrak{I}'(Sp(n))$, both $m_\xi^+$ and $m_\xi^-$ defined by \eqref{m+m-} are necessarily even.
 For each $m\leq n$, let ${G}_{m}^{\mathbbm{H}}(n)$ be
the Grassmannian of
  $J$-stable vector subspaces of $\C^{2n}\cong\mathbbm{H}^{n}$ with quaternionic dimension
  $m$). Assume that $m_\xi^+\leq m_\xi^-$ (if $m_\xi^+> m_\xi^-$, then $N_\xi=-N_{\xi'}$ for some $\xi'$ with $m_{\xi'}^+\leq m_{\xi'}^-$, as in Remark  \ref{togeodequivalent}).
  We identify the inner symmetric space $N_\xi=\{g\gamma_\xi(-1)g^{-1}\}$  with the quaternionic Grassmannian
$G^{\mathbbm{H}}_{m}(n)$, where $2m=m_\xi^+$,  via the totally geodesic immersion
$\iota:G^{\mathbbm{H}}_{m}(n)\hookrightarrow \sqrt{e}\subset Sp(n)$ defined by $\iota(V)=\pi_V^\perp-\pi_V$.
\begin{prop}
  Given a symmetric canonical element $\xi$ of $Sp(n)$, if $N_\xi\cong G^{\mathbbm{H}}_{m}(n)$,  with $2m\leq n$, then
  $r_{\rho_k}(\xi)\leq 4mk$ if $k\leq n-2{m}+1$, and $$r_{\rho_k}(\xi)\leq 4m(n-2m+1)+(k-n+2m-1)(2m+n-k)$$ if $k> n-2m+1$. In particular, for the standard representation, we have  $r_{\rho_1}(\xi)\leq 4m$ and for the adjoint representation we have $r_{\rho_2}(\xi)\leq 8m$ if $2m<n$ and  $r_{\rho_2}(\xi)\leq 8m-2$ if $2{m}=n$.
\end{prop}
\begin{proof}
  The symmetric canonical element $\xi$ associated to $\iota(G^{\mathbbm{H}}_{m}(n))$ that maximizes the uniton number $r_{\rho_k}(\xi)$ is the element $\xi_I+\delta_I(n)H_n$ with  $I=\{n-2m+1, n-2m+2,\ldots, n\}$.
\end{proof}
The estimation $r_{\rho_1}(\xi)\leq 4m$ we have obtained  for the minimal uniton number  of harmonic maps $\varphi:M\to \iota(G^{\mathbbm{H}}_{m}(n))$ coincides with that of \cite{pacheco_2006}, Theorem 4.

\subsubsection{Quaternionic Projective Space.}
 The canonical symmetric elements associated to  $\mathbbm{H}P^{n-1}$ are precisely $\xi_1=H_1$, $\xi_i=H_{i-1}+H_{i}$ (with $i=2,\ldots,n-1$) and $\xi_n=H_{n-1}+2H_n$. On the other hand, we have the following characterization for such harmonic maps, which generalises the work of  Aithal \cite{Ai_1986}.
\begin{thm} \cite{pacheco_2006}\label{classteo}
{Let $M$ be a Riemann surface and $A \subset B$ be two holomorphic
$J$-isotropic vector subbundles of $M\times \C^{2n}$ such that
$\dim B\ominus A=1$. Suppose also that, for any complex coordinate system $(U,z)$,  $\frac{\partial}{\partial z} A\subset B$ and
$\frac{\partial}{\partial z} B\subset JB^\perp$. Let $R$ be a holomorphic vector subbundle
of $A\oplus JB$, with $\dim{R}=\dim{A}$, such that: a) $
\frac{\partial}{\partial z} R\perp JR$ (in
particular $R$ is isotropic); b) $\pi_A(R)=A$ almost everywhere.
Then
\begin{equation}
\label{phihar} {\varphi}=(B\oplus JB)\ominus (R\oplus JR)
\end{equation}
gives a harmonic map from $M$ with values in $\mathbbm{H}P^{n-1}$.
Conversely, any harmonic map
$\varphi:M\rightarrow\mathbbm{H}P^{n-1}$ of finite uniton number arises in this
 way.}
\end{thm}
We observe that the harmonic map $\varphi$ given by \eqref{phihar} is associated to the symmetric canonical element $\xi_i$  if and only if the rank of $B$ is $i$, with $i=1,\ldots, n$.
 Moreover, when $R=A$, $\varphi$ corresponds to an $S^1$-invariant extended solution. In such case, we have $\varphi=f\oplus Jf$, with $f=B\ominus A:M\to \mathbbm{C}P^{2n-1}$ an harmonic map.

 It is clear from Lemma \ref{lemata} that any extended solution with values in $U_{\xi_1}$ must be $S^1$-invariant, and consequently the corresponding harmonic maps are of the form $\varphi=f\oplus Jf$, with $f:M\to \mathbbm{C}P^{2n-1}$
holomorphic. Such harmonic maps are called   \emph{quaternionic mixed pairs} in  \cite{BW_1991}.

The harmonic maps corresponding to $\xi_n$ are given by $\varphi=(R\oplus JR)^\perp$, where $R$ is an holomorphic vector subbundle of $M\times\mathbbm{C}^{2n}$ with $\mathrm{rank}\,R= n-1$.  Set $h=R^\perp\cap\frac{\partial}{\partial z}R$, and consequently $Jh=JR^\perp\cap\frac{\partial}{\partial \bar z}R$. Since $\frac{\partial}{\partial z}R\perp JR$,  we have $h\perp R\oplus JR$. Hence $\varphi=h\oplus Jh$, with $Jh=h^\perp\cap \frac{\partial}{\partial z}h$, that is, $\varphi$ is a \emph{quaternionic Frenet pair} \cite{BW_1991}.

The canonical factorizations of extended solutions $\Phi$ associated to filtrations $H_i+H_{i-1}\preceq H_{i}$ and $H_{n-1}+2H_n\preceq 2H_{n}$ are defined by \eqref{wi} and can be written, in terms of the holomorphic data of Theorem \ref{classteo}, as
\begin{equation}\label{facsp}
\Phi=(\lambda^{-1}\pi_B+\pi_{(B\oplus JB)^\perp}+\lambda\pi_{JB})(\lambda^{-1}\pi_R+\pi_{(R\oplus JR)^\perp}+\lambda\pi_{JR}).\end{equation}
In order to prove formula \eqref{facsp},  recall that Proposition 4 of \cite{pacheco_2006} gives a description of $W=\Phi H_+^{2n}$.
Using this description, apply \eqref{wi} to see that
$$W^1=\Phi_1H_+^{2n}=B\lambda^{-1}+JB^\perp+\lambda H_+^{2n}.$$
Hence, $\Phi_1= \lambda^{-1}\pi_B+\pi_{(B\oplus JB)^\perp}+\lambda \pi_{JB}$. Finally,
check that $$\Phi_1^{-1}W=(\lambda^{-1}\pi_R+\pi_{(R\oplus JR)^\perp}+\lambda\pi_{JR})H_+^{2n}.$$

\subsubsection{The space of complex Lagrangian subspaces.} This is the symmetric space  $\mathcal{L}_n:=Sp(n)/U(n)$, which has dimesion $n(n+1)$. The corresponding involution $\sigma$, interpreting $Sp(n)$ as the group of quaternionic  $n\times n$-matrices  $X$  satisfying $XX^*=Id$, is given by $\sigma(X=A+jB)=A-jB$. Then $\sigma=Ad(s_0)$ where $s_0=\sqrt{-1}Id\in Sp(n)$. This element does not square to the identity ($s_0^2=-Id$), which means that we have to embed $\mathcal{L}_n$ in the adjoint group $Ad(Sp(n))$. The symmetric canonical elements associated to $\mathcal{L}_n$ are precisely the elements of the form $\xi_I=\sum_{i\in I}H_i$ with $n\in I$. To check this one can compute the dimension of  $N_{\xi_I}$  by using  \eqref{dimxi} and identify  those $\xi_I$ for which $\dim N_{\xi_I}=n(n+1)$.

In  \cite{svensson_wood_2010}, Section 6.8, the authors use an alternative approach via loop groups to study harmonic maps of finite type into $\mathcal{L}_n$: any such map $\varphi: M\to \mathcal{L}_n$, considering $\mathcal{L}_n$ as a totally geodesic submanifold of $G_n(\C^{2n})$, the Grassmannian of $n$-dimensional complex subspaces of $\C^{2n}$, admits a \emph{symplectic} $\mathcal{I}$-invariant extended solution, that is an  $\mathcal{I}$-invariant extended solution of the form $\Phi=\sum_{i=0}^rA_i\lambda^i$, for some odd positive integer $r$, with values in $\Omega U(2n)$ which satisfies $J\Phi J^{-1}=\lambda^{-r}\Phi$. It is easy to check that, in this case,
$$\tilde{\Phi}:=Ad(\lambda^{-r/2}\Phi):M\to \Omega Ad(Sp(n))$$ is an extended solution  associated to the same harmonic map $\varphi:M\to \mathcal{L}_n\subset Ad(Sp(n))$.

The Example 6.31 in \cite{svensson_wood_2010} exhibits a class of harmonic maps $M\to \mathcal{L}_3$ whose associated  $S^1$-invariant symplectic extended solutions are of the form
$$\Phi=\pi_A+\lambda \pi_B+\lambda^2\pi_{JB}+\lambda^3 \pi_{JA},$$
where $A$ and $B$  are $J$-isotropic holomorphic vector subbundles of $M\times \C^6$, with ranks 1 and 2, respectively, and orthogonal projections $\pi_A$ and $\pi_B$. Then
$$\tilde{\Phi}:=Ad({\lambda^{-3/2}\pi_A+\lambda^{-1/2}\pi_B+\lambda^{1/2}\pi_{JB}+\lambda^{3/2} \pi_{JA}}),$$
which corresponds to the symmetric canonical element $\xi=H_1+H_3=\frac32E_1+\frac12(E_2+E_3)$.

\subsection{Canonical elements of $(D_n)$}

The roots are the vectors $\pm L_i\pm L_j$, with $i\neq j$ and $1\leq i,j\leq n$. Fix the positive root system
$\Delta^+=\{L_i\pm L_j\}_{i<j}.$ The positive simple roots are then the roots $\alpha_i=L_i-L_{i+1}$, for $1\leq i\leq n-1$, and $\alpha_n=L_{n-1}+L_n$. The fundamental weights are
$\omega_i=L_1+\ldots+L_i$, for $1\leq i\leq n-2$, and  $$\omega_{n-1}=\frac12(L_1+\ldots+L_{n-1}-L_n),\, \omega_n=\frac12(L_1+\ldots+L_n).$$
  The dual basis $H_1,\ldots,H_n\in \lt$ of  $\alpha_1,\ldots,\alpha_n\in \sqrt{-1}\lt^*$ is given by $H_i  =E_1+E_2+\ldots+E_i$, for $1\leq i\leq n-2$, and
\begin{align*} H_{n-1}  =\frac12(E_1+E_2+\ldots+E_{n-1}-E_n),\, H_n  =\frac12(E_1+E_2+\ldots+E_{n-1}+E_n). \end{align*}
In terms of the dual basis $\eta_1,\ldots,\eta_n\in \lt$  of $\omega_1,\ldots,\omega_n\in \sqrt{-1}\mathfrak{t}^*$, these elements can be written as follows
\begin{align*} H_i & =\eta_1+2\eta_2+\ldots+i\eta_i+\ldots+i\eta_{n-2}+\frac{i}{2}(\eta_{n-1}+\eta_n) \,\,\,\mbox{for $i\leq n-2$}\\ H_{n-1} & =\frac12(\eta_1+2\eta_2+\ldots+(n-2)\eta_{n-2})+\frac{n}{4}\eta_{n-1}+\frac{n-2}{4}\eta_n\\ H_n &= \frac12(\eta_1+2\eta_2+\ldots+(n-2)\eta_{n-2})+\frac{n-2}{4}\eta_{n-1}+\frac{n}{4}\eta_n. \end{align*}
The simply connected Lie group with Lie algebra $\g=\so(2n)$ is $\tilde G=Spin(2n)$ and its centre $Z(Spin(2n))$ is $\Z_4$ when $n$ is odd and $\Z_2\oplus\Z_2$ when $n$ is even. We use the same notation as in \cite{fulton_harris} to write, in the case $n$ even, $Z(Spin(2n))=\{\pm 1,\pm w\}$.  The standard representation has highest weight $\omega_1$ and the adjoint representation has highest weight $\omega_2$.

 Given $I\subseteq \{1,\ldots,n\}$, let $I_0=I\cap\{1,\ldots,n-2\}$. For each $i=0,1,2$, define $\varepsilon^i_I=1$ if $|I\cap\{n-1,n\}|=i$ and  $\varepsilon^i_I=0$ otherwise. The next theorem describes all the $I$-canonical elements of $Spin(2n)$.
\begin{thm}\label{canspin}    When $n$ is odd, the $I$-canonical elements of $Spin(2n)$ are precisely the elements
 \begin{itemize}
    \item[a)] $\xi_I+3\varepsilon_I^1H_s$ and $\xi_I+\varepsilon_I^1(H_s+H_j)$ if $\sum_{i\in I_0}i$ is even;    \item[b)]  $\xi_I+\varepsilon_I^0H_j+\varepsilon_I^1H_s+2\varepsilon_I^2H_s$ and  $\xi_I+\varepsilon_I^0H_j+\varepsilon_I^1H_s+\varepsilon_I^2H_j$ if $\sum_{i\in I_0}i$ is odd;
  \end{itemize}
where $j\in I_0$ is odd and $s\in I\cap\{n-1,n\}$. When $n$ is even, the $I$-canonical elements of $Spin(2n)$ are precisely the elements
   \begin{itemize}
    \item[a)] $\xi_I+\varepsilon_I^1 H_s+\varepsilon_I^2 H_j$ and $\xi_I+\varepsilon_I^1 H_s+\varepsilon_I^2 (H_{n-1}+H_n)$ if $\sum_{i\in I_0}i$ is even;
    \item[b)] $\xi_I+\varepsilon_I^0 H_j+\varepsilon_I^1 (H_s+H_j)$ if $\sum_{i\in I_0}i$ is odd;
  \end{itemize}
where $j\in I_0$ is odd and $s\in I\cap\{n-1,n\}$.
\end{thm}
The fundamental weights $\omega_n,\omega_{n-1}$  correspond to the \emph{half-spinor}  representations $S^+,S^-$ of $\mathfrak{so}(2n,\C)$ (see \cite{fulton_harris} for details): write $(\mathbbm{R}^{2n})^\C=U\oplus \overline{U}$, where $U$ is a $n$-dimensional isotropic subspace; the half-spinor representations are  $S^+=\wedge^{\mathrm{even}} U$ and $S^-=\wedge^{\mathrm{odd}} U$, with $\mathfrak{so}(2n,\C)$ acting as in the construction of the spinor representation of  $\mathfrak{so}(2n+1,\C)$.
 For $n=3$, the hal-spinor representation $S^+$ gives the identification $Spin(6)\cong SU(4)$. From Theorem \ref{canspin} we see that
  the non-trivial elements of $Spin(6)$, up to automorphisms of the corresponding Dynkin diagram, are those shown in the Table \ref{cococo}.
  \begin{table}[!htb]
\begin{tabular}{c|c|c|c} $Spin(6)$ & $|I|=3$ & $|I|=2$ & $|I|=1$ \\ \hline  &$H_1+2H_2+H_3$ & $2H_1+H_2$ & $4H_1$  \\ & $3H_1+H_2+H_3$  & $H_1+H_3$ & $2H_2$
\end{tabular}
 \vspace{.10in}
\caption{Canonical elements for $Spin(6)\cong SU(4)$.}\label{cococo}
\end{table}
This agrees, up to label, with the results of Section 5.5 of \cite{correia_pacheco_4}, where the reader can find a description of the corresponding $S^1$-invariant extended solutions with respect to the standard representation of $SU(4)$.

The representation of $Spin(2n)$ with highest weight $\omega^*=\sum \lambda_iL_i$ descends to a representation of $SO(2n)\cong Spin(2n)/\Z_2$ if all $\lambda_j$ are integers.
 \begin{prop}
   The unique $I$-canonical element of $SO(2n)$ is $\xi_I+\varepsilon_I^1 H_{s},$ with $s\in \{n-1,n\}\cap I$. Moreover,  $r_{\rho_1}(\xi_I+\varepsilon_I^1 H_{s})\leq 2n-2$  for the standard representation  and $r_{\rho_2}(\xi_I+\varepsilon_I^1 H_{s})\leq 4n-6$ for the adjoint representation.
 \end{prop}
\begin{proof}
Whether $n$ is even or odd,  the canonical elements of $\SO(2n)$ are the maximal  integral combinations of the vectors $H_i$ which are simultaneously integral linear combinations of the vectors $E_i$.
\end{proof}
This minimal uniton estimation for the standard representation agrees with that of Proposition 6.17 in \cite{svensson_wood_2010}.
Next we analyse the canonical elements for the remaining quotient groups in the case $n$ is even.
\begin{thm}  Assume  $n$ is  even. We have:
   \begin{itemize}
    \item[a)] if $\sum_{i\in I_0}i$ is even, the $I$-canonical elements of $Spin(2n)/\{1,{w}\}$ are precisely the elements
    $\xi_I+\delta_I(n-1)H_{n-1}$ and  $\xi_I+\delta_I(n-1)H_{j},$ with $j\in I_0$ odd.
    \item[b)] if $\sum_{i\in I_0}i$ is odd, the  $I$-canonical elements of $Spin(2n)/\{1,\mathrm{w}\}$ are precisely the elements
    $\xi_I+\varepsilon^0_{I}H_j+\varepsilon^1_{I}\delta_I(n)H_j,$ with $j\in I_0$ odd.
  \end{itemize}
 Moreover, if $n=4m$, then the $I$-canonical elements of  $Spin(2n)/\{1,{-w}\}$ are precisely the $I$-canonical elements of $Spin(2n)$. If $n=4m+2$, the $I$-canonical elements of  $Spin(2n)/\{1,{-w}\}$ are precisely the $I$-canonical elements of $Spin(2n)/\{1,{w}\}$.
\end{thm}
\begin{proof}
  The representation of $Spin(2n)$ with highest weight $\omega^*=\sum \lambda_iL_i$ is a representation of the quotient group $Spin(2n)/\{1,{w}\}$ if $\sum\lambda_i$ is an even integer and
  $$\lambda_1\geq \lambda_2\ldots \geq \lambda_{n-1}\geq |\lambda_n|\geq 0,$$
  with the $\lambda_i$ all  half-integers (in this case $\lambda_n\neq 0$) or all integers (conf. \cite{fulton_harris}). Hence,  taking account Theorem \ref{repG}, we can check that the elements $H_n$ and $H_j+H_{n-1}$, with $j\in I_0$ odd, are in $\mathfrak{I}(Spin(2n)/\{1,{w}\})$, since $\omega^*(H_n),\omega^*(H_j+H_{n-1})\in\mathbbm{Z}$ for any such representation $\omega^*$. On the other hand, $H_{n-1}$, $H_j$, with $j\in I_0$ odd, and $H_{n-1}+H_n$ are not in $\mathfrak{I}(Spin(2n)/\{1,{w}\})$.
In particular,  the $I$-canonical elements of $Spin(2n)/\{1,{w}\}$ must be of the form $\xi=\sum_{i\in I}n_iH_i$, with $n_i=1$ if $i$ is even or $i=n$, and $1\leq n_i\leq 2$ otherwise.

The representation of $Spin(2n)$ with highest weight $\omega^*=\sum \lambda_iL_i$ is a representation of the quotient group $Spin(2n)/\{1,-\mathrm{w}\}$ if $\sum\lambda_j-\frac{n}{2}$ is an odd integer. If $n=4m+2$, this condition is equivalent to say that  $\sum\lambda_i$ is an even integer. Hence the $I$-canonical elements are the same as in the $Spin(2n)/\{1,w\}$ case. If $n=4m$, one can check that none of the elements  $H_j$, $H_{n-1}$, $H_n$, $H_j+H_{n-1}$, $H_j+H_{n}$ and $H_{n-1}+H_n$ is in $\mathfrak{I}(Spin(2n)/\{1,{-w}\})$. Hence, in this case, the $I$-canonical elements   are precisely the $I$-canonical elements of $Spin(2n)$.
 \end{proof}
\begin{rem}If $n$ is even and, for example, $I=\{1,3,n-1\}$, $\xi'=H_1+2H_3+H_{n-1}$ is a $I$-canonical element of $Spin(2n)/\{1,{w}\}$.
 On the other hand, by Theorem \ref{canspin}, the unique $I$-canonical elements of $Spin(2n)$ is $H_1+H_3+2H_{n-1}$. Hence
 there exists no $I$-canonical element $\xi$ of $Spin(2n)$
such that $\xi'\preceq \xi$, as we have mentioned in Remark \ref{remark}.
\end{rem}
\subsubsection{Symmetric canonical elements of $(D_n)$}
Let us describe the symmetric canonical elements of $SO(2n)$.
\begin{thm}\label{smhp1} Given $I\subseteq \{1,\ldots,n\}$, let $I_0=I\cap\{1,\ldots,n-2\}$.
The symmetric canonical elements of ${SO}(2n)$ are precisely the elements of the form $\xi_I+\varepsilon_I^1H_{s}$ and  $\xi_I+\varepsilon_I^1H_{s}+2\varepsilon^2_I H_{s}$, with $s\in\{n-1,n\}\cap I$.
\end{thm}

Given $\xi\in \mathfrak{I}'(SO(2n))$, define $m_\xi^+$ and $m_\xi^-$ as in \eqref{m+m-}.
In this case, both numbers $m_\xi^+$ and $m_\xi^-$ are even.
The inner symmetric space $N_\xi=\{g\gamma_\xi(-1)g^{-1}\}$ can be identified with the real Grassmannian
$G^{\R}_{m}(2n)$, where $m=\min \{m_\xi^+,m_\xi^-\}$,  via the totally geodesic  immersion
$\iota:G^{\R}_{m}(2n)\hookrightarrow \sqrt{e}\subset SO(2n)$  defined by $\iota(V)=\pi_V^\perp-\pi_V$ if $m=m_\xi^+$ and $\iota(V)=\pi_V-\pi_V^\perp$ if $m=m_\xi^-$.
\begin{thm}
  Given $m\leq n$ ($m$ even),  then
  $r_{\rho_k}(G^{\R}_m(2n))\leq 2(m-1)k$ if $k\leq n-m+1$, and $$r_{\rho_k}(G^{\R}_m(2n))= 2(m-1)(n-m+1)+(k-n+m-1)(n+m-k)+2$$ if $n-1 > k> n-m+1$. In particular, for the standard representation we have  $r_{\rho_1}(G^{\R}_m(2n))= 2(m-1)$ and for the adjoint representation we have $r_{\rho_2}(G^{\R}_m(2n))=4m-4$.
\end{thm}
\begin{proof}
Since $\exp(2H_{n-1}\pi)=\exp(2H_{n}\pi)=-Id\in Z(\SO(2n))$, in view of Remark \ref{togeodequivalent} and  Theorem \ref{smhp1},
 the symmetric canonical element  associated to $G^{\R}_m(2n)$ that maximizes  $r_{\rho_k}(\xi)$ is the element $\xi_I=\sum_{i\in I}H_i$ with  $I=\{n-m+1, n-m+2,\ldots, n\}$.
\end{proof}

Again, this agrees with the estimation of the minimal uniton number of harmonic maps into $G^{\R}_m(2n)$, with $m$  even, given in Proposition 6.23 of \cite{svensson_wood_2010}

The symmetric canonical elements associated to $G^{\R}_2(2n)$ are the elements $\xi_1=H_1$, $\xi_i=H_i+H_{i-1}$, with  $i=2,\ldots, n-2$, $\xi_{n-1}=H_{n-2}+H_{n-1}+H_n$ and $\xi_n=H_n+ H_{n-1}$. We have $r_{\rho_1}(\xi_1)=r_{\rho_1}(\xi_n)=2$, and $r_{\rho_1}(\xi_i)=4$ for $i=2,\ldots, n-1$.
 Harmonic maps into $G^{\R}_2(2n)$  corresponding $\xi_1$  are of the form $\varphi=f\oplus \bar{f}$, with $f:M\to \mathbbm{R}P^{2n+1}$ holomorphic, that is, $\varphi$ is a \emph{real mixed pair} \cite{Bahy_Wood_1989}. Those harmonic maps corresponding to $\xi_n$ admit extended solutions of the form
 $W=A_{-1}\lambda^{-1}+\overline{A}_{-1}^\perp+\lambda H_+^{2n},$
  with $A_{-1}$ an holomorphic isotropic subbundle of $M\times \C^{2n}$ of rank $n-1$. Here, $m_\xi^+=2n-2$ and $m_\xi^-=2$. Hence
 $\varphi=\overline{A}_{-1}^\perp\ominus A_{-1}$. The $S^1$-invariant extended solutions associated to the  canonical elements $\xi_i$, with $i=2,\ldots, n-1$,
 produce harmonic maps of the form $\varphi=h\oplus \bar h$, with $h= A_{-1}\ominus A_{-2}$, where  $A_{-2}$ and $A_{-1}$ are holomorphic isotropic subbundles of ranks $i-1$ and $i$, respectively.

\subsubsection{Space of positive orthogonal complex structures.}

A positive orthogonal complex structure of $\mathbbm{R}^{2n}$ is an orientation preserving isometry $J$ of $\mathbbm{R}^{2n}$ with $J^2=-Id$. The space $J_+(\mathbbm{R}^{2n})$ of all positive orthogonal complex structures can be totally geodesically immersed in $\sqrt{e}\subset Ad(SO(2n))\cong SO(2n)/\mathbbm{Z}_2$ via $J\mapsto \mathrm{Ad}(J)$. Taking account \eqref{Nxi}, the element $\xi_I\in\mathfrak{I}(SO(2n)/\mathbbm{Z}_2)$ corresponds to $J_+(\mathbbm{R}^{2n})$ if, and only if, $|I\cap \{n-1,n\}|=1$. From this, $r_{\rho_2}(\xi_I)\leq 4n- 8$. Moreover, given a such canonical element $\xi_I$ and an extended solution $\Phi:M\to U^{\mathcal{I}}_{\xi_I}(SO(2n)/\mathbbm{Z}_2)$, the filtration defined by \eqref{canfac} induces
 a  factorization $\Phi=\Phi_{l-1}(\Phi_{l-1}^{-1}\Phi_{l-2})\ldots (\Phi_{1}^{-1}\Phi)$ where each $\Phi_i$ corresponds to an harmonic map into $J_+(\mathbbm{R}^{2n})$.

\section{Canonical elements of exceptional Lie algebras}

When the Lie algebra $\g$ is the compact real form  of one of the exceptional Lie algebras $\g_2$, $\mathfrak{f}_4$ and $\mathfrak{e}_8$, the corresponding simply connected Lie group has trivial centre and the $I$-canonical element is $\xi_I=\sum_{i\in I}H_i$. Harmonic maps of finite uniton number into $G_2$ and its unique inner symmetric space $G_2/SO(4)$ were studied in \cite{correia_pacheco_3}, where a description of the extended solutions associated to the different canonical elements was given in terms of \emph{Frenet frames} (see also \cite{svensson_wood_2013} for a different twistorial approach to harmonic maps into $G_2/SO(4)$).
Next we focus our attention in the symmetric canonical elements of $F_4$, which is simply connected and has trivial centre. The extended Dynkin diagram of the $52$-dimensional Lie algebra  $\mathfrak{f}_4$ is given by
\begin{center}
\psset{xunit=1.0cm,yunit=1.0cm,algebraic=true,dotstyle=o,dotsize=3pt 0,linewidth=0.5pt,arrowsize=3pt 2,arrowinset=0.25}
\begin{pspicture*}(-3,-1)(3,0.5)
\psline(-1,0)(0,0)
\psline(1,0)(2,0)
\psline(0,-0.04)(1,-0.04)
\psline(0,0.03)(1,0.03)
\psline(-1.94,0)(-1,0)

\psline(0.35,0.2)(0.55,0)
\psline(0.35,-0.2)(0.55,0)

\rput[tl](-1.06,-0.24){$\alpha_1$}
\rput[tl](-0.06,-0.23){$\alpha_2$}
\rput[tl](0.93,-0.23){$\alpha_3$}
\rput[tl](1.95,-0.24){$\alpha_4$}
\rput[tl](-2.11,-0.23){$-\theta$}
\begin{scriptsize}
\psdots[dotsize=3pt 0,dotstyle=*](-1,0)
\psdots[dotsize=3pt 0,dotstyle=*](0,0)
\psdots[dotsize=3pt 0,dotstyle=*](1,0)
\psdots[dotsize=3pt 0,dotstyle=*](2,0)
\psdots[dotsize=3pt 0](-2,0)
\end{scriptsize}
\end{pspicture*}
\end{center}
where $\theta=2\alpha_1+3\alpha_2+4\alpha_3+2\alpha_4$ is the highest root. The inner involutions of a simple Lie algebra were classified by Borel and Siebenthal \cite{borel_siebenthal_1949} (c.f. Theorem 3.17 of \cite{burstall_rawnsley_1990}). In the $\f_4$ case, there are two conjugation classes of inner involutions, $\tau_1=Ad({\exp{\pi H_1}})$ and  $\tau_4=Ad({\exp{\pi H_4}})$, which correspond to the symmetric spaces $F_4/Sp(3)Sp(1)$ and $\mathbbm{O}P^2\cong F_4/Spin(9)$ (the \emph{Cayley plane}), respectively. The isotropy subgroup $Spin(9)$ has dimension $36$ and its Lie subalgebra, being the $1$-eigenspace of $\tau_4$, is given by $$\mathfrak{spin}(9)\cong\mathfrak{t}^\C+\!\!\!\!\! \bigoplus_{\text{$\alpha(H_4)$  even}}\!\!\!\!\!\mathfrak{g}_\alpha.$$
Hence, given $I\subset\{1,2,3,4\}$, $N_{\xi_I}\cong F_4/Spin(9)$ if   $$\dim\bigoplus_{\text{$\alpha(\xi_I)$ odd}}\!\!\!\mathfrak{g}_\alpha=16$$ and
$N_{\xi_I}\cong F_4/Sp(3)Sp(1)$ otherwise. From this, one can check that the symmetric canonical elements of $F_4$ associated to the Cayley plane are precisely the elements $H_3$, $H_4$ and $H_3+H_4$. The lowest dimensional representation $\rho$ of $F_4$ has highest weight $\omega^*=\alpha_1+2\alpha_2+3\alpha_3+2\alpha_4$ and dimension 26. Hence  $r_{\rho}(H_4)=4$, $r_{\rho}(H_3)=6$ and $r_{\rho}(H_3+H_4)=10$. The representation $\rho$ arises from the identification of $F_4$ with the group of automorphisms of the $27$-dimensional exceptional Jordan algebra $\mathfrak{h}_3(\mathbbm{O})$ (c.f. \cite{Baez}). Being the group of automorphisms, it preserves the identity element $\mathrm{1}$ and the natural inner product on  $\mathfrak{h}_3(\mathbbm{O})$, which is defined via the trace and the multiplication of elements, and consequently $F_4$ acts in $\mathrm{1}.\mathbbm{R}^\perp$.  Hence, by using the representation $\rho$, the Cayley plane can be immersed  in some connected component of $\sqrt{e}\subset SO(26)$, that is, in some real Grassmannian. In \cite{copa} we exploit this point of view to study harmonic spheres and tori in $\mathbbm{O}P^2$. Since $F_4$ has trivial centre, the minimal uniton number of $F_4$ with respect to any representation  is attained by the canonical element $H_1+H_2+H_3+H_4$.

The compact simply connected real Lie group  $\tilde{E}_7$ associated to  the 133-dimensional  Lie algebra  $\mathfrak{e}_7$  has centre $\mathbbm{Z}_2$. The corresponding extended Dynkin diagram is given by
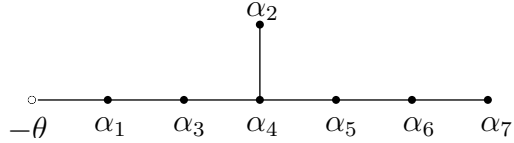
\begin{center}
  \psset{xunit=1.0cm,yunit=1.0cm,algebraic=true,dotstyle=o,dotsize=3pt 0,linewidth=0.5pt,arrowsize=3pt 2,arrowinset=0.25}
\begin{pspicture*}(-3.5,-.5)(3.5,1.55)
\psline(-2,0)(-1,0)
\psline(-1,0)(0,0)
\psline(0,0)(1,0)
\psline(1,0)(2,0)
\psline(2,0)(3,0)
\psline(0,0)(0,1)
\rput[tl](-0.18,-0.23){$\alpha_4$}
\rput[tl](0.86,-0.23){$\alpha_5$}
\rput[tl](1.88,-0.23){$\alpha_6$}
\rput[tl](2.9,-0.23){$\alpha_7$}
\rput[tl](-1.14,-0.23){$\alpha_3$}
\rput[tl](-2.18,-0.23){$\alpha_1$}
\rput[tl](-0.18,1.3){$\alpha_2$}
\rput[tl](-3.32,-0.22){$-\theta$}
\psline(-2,0)(-2.92,0)
\begin{scriptsize}
\psdots[dotsize=3pt 0,dotstyle=*](0,0)
\psdots[dotsize=3pt 0,dotstyle=*](1,0)
\psdots[dotsize=3pt 0,dotstyle=*](2,0)
\psdots[dotsize=3pt 0,dotstyle=*](3,0)
\psdots[dotsize=3pt 0,dotstyle=*](-1,0)
\psdots[dotsize=3pt 0,dotstyle=*](-2,0)
\psdots(-3,0)
\psdots[dotstyle=*](0,1)
\end{scriptsize}
\end{pspicture*}
\end{center}
where $\theta=\omega_1=2\alpha_1+2\alpha_2+3\alpha_3+4\alpha_4+3\alpha_5+2\alpha_6+\alpha_7$ is the highest weight, and the duals $H_i$ of the simple roots are related with the duals $\eta_i$ of the fundamental weights  by
$$\left[H_i\right]=\left[\begin{array}{ccccccc} 2 & 2 & 3 & 4 & 3 & 2 & 1\\ 2 & 7/2 & 4 & 6 & 9/2 & 3 & 3/2\\ 3 & 4 & 6 & 8 & 6 & 4 & 2\\ 4 & 6 & 8 & 12 & 9 & 6 & 3\\ 3 & 9/2 & 6 & 9 & 15/2 & 5 & 5/2\\ 2 & 3 & 4 & 6 & 5 & 4 & 2\\ 1 & 3/2 & 2 & 3 & 5/2 & 2 & 3/2 \end{array}\right]\left[\eta_i\right].$$
The lowest dimensional representation of  $\tilde{E}_7$, which does not descend to a representation of the adjoint group  $E_7=\tilde{E}_7/\mathbbm{Z}_2$, has highest weight $$\omega_7=\alpha_1+3/2\alpha_2+2\alpha_3+3\alpha_4+5/2\alpha_5+2\alpha_6+3/2\alpha_7$$ and dimension 56.
Given $I\subset\{1,\ldots,7\}$, define $\varepsilon^i_I=1$ if $|I\cap\{2,5,7\}|=i$ and  $\varepsilon^i_I=0$ otherwise. The $I$-canonical elements of $\tilde{E}_7$ are the elements
$\sum_{i\in I}H_i+\varepsilon^1_I H_s+ \varepsilon^3_IH_s$, with $s\in I\cap\{2,5,7\}$. Hence,  we have $r_{\rho_1}(\tilde E_7)= r_{\rho_1}(\sum_{i=1}^7H_i+H_5)= 40$ (adjoint representation) and $r_{\rho_7}(\tilde E_7)= r_{\rho_7}(\sum_{i=1}^7H_i+H_5)= 32 $ (lowest dimensional representation).

 For completeness, we present in Table 3 the minimal uniton number for harmonic maps into the exceptional Lie groups with respect to the lowest dimensional representation $\rho_L$ and to the adjoint representation $\rho_A$  of the corresponding Lie algebra. The lowest dimensional representation $\rho_L$ of $\mathfrak{e}_6$ and $\mathfrak{e}_7$ does not descend to representations of $E_6$ and $E_7$. When $G$ has trivial centre, the minimal uniton numbers with respect to $\rho_A$ were already established in \cite{burstall_guest_1997}.
 The minimal uniton number for harmonic maps into the inner symmetric spaces of exceptional Lie groups with trivial centre is given in Table 4. For $\mathfrak{e}_8$, we have $\rho_L=\rho_A$. In the  $\mathfrak{e}_6$ case, and viewing the corresponding Dynkin diagram as the subdiagram with nodes $\alpha_1,\ldots,\alpha_6$ of the Dynkin diagram  of $\mathfrak{e}_7$,
   the elements $H_i$ are related with the duals $\eta_i$ of the fundamental weights  by $$\left[H_i\right]=\left[\begin{array}{cccccc} 4/3 & 1 & 5/3 & 2 & 4/3 & 2/3\\ 1 & 2 & 2 & 3 & 2 & 1\\ 5/3 & 2 & 10/3 & 4 & 8/3 & 4/3\\ 2 & 3 & 4 & 6 & 4 & 2\\ 4/3 & 2 & 8/3 & 4 & 10/3 & 5/3\\ 2/3 & 1 & 4/3 & 2 & 5/3 & 4/3 \end{array}\right]\left[\eta_i\right].$$ The compact simply connected real Lie group  $\tilde{E}_6$  has centre isomorphic to $\Z_3$.
There are two lowest dimensional representations  of $\tilde{E}_6$: those associated to the weights  $\omega_1$ and $\omega_6$. We fix $\rho_L=\rho_{\omega_1}$. The minimal uniton numbers  $r_{\rho_L}(\tilde{E}_6)$ and $r_{\rho_A}(\tilde{E}_6)$ are attained by the canonical element $\xi=2H_1+H_2+H_3+H_4+2H_5$.

\begin{table}[!htb]
\begin{tabular}{ c| c| c| }
 {$G$} & {$r_{\rho_L}(G)$} & {$r_{\rho_A}(G)$} \\

 \hline
{$G_2$} & {6} & {10}   \\

 \hline
{$F_4$}  & {16} & {22}  \\

\hline {$\tilde{E}_6$} &{20} & {26} \\ {${E}_6$} & {$-$} &{22} \\

 \hline {$\tilde{E}_7$} & {32} & {40}  \\ {${E}_7$} & {$-$} &{34} \\

\hline {$E_8$} & {58} & {58}   \end{tabular}\vspace{.10in}
\caption{Minimal uniton numbers for harmonic maps into exceptional Lie groups.}
\end{table}
\begin{table}[!htb]
\begin{tabular}{ c| c| c| }
{$G/H$} & {$r_{\rho_L}(G/H)$} & {$r_{\rho_A}(G/H)$} \\

 \hline
{$\frac{G_2}{SO(4)}$} & {6} & {10}   \\

 \hline
{$\frac{F_4}{Spin(9)}$}  & {10} & {12}  \\

{$\frac{F_4}{Sp(3)Sp(1)}$}  & {16} & {22}  \\

\hline {$\frac{E_6}{Spin(10)U(1)}$} & {$-$} & {16} \\

{$\frac{E_6}{SU(6)Sp(1)}$} & {$-$} & {22} \\

 \hline {$\frac{E_7}{E_6U(1)}$} & {$-$} & {24} \\
 {$\frac{E_7}{Spin(12)Sp(1)}$} & {$-$} & {32} \\
 {$\frac{E_7}{SU(8)/\mathbbm{Z}_2}$} & {$-$} & {34} \\

\hline {$\frac{E_8}{E_7Sp(1)}$} & {48} & {48} \\
{$\frac{E_8}{Spin(16)/{\mathbbm{Z}}_2}$} & {58} & {58}  \end{tabular}
\vspace{.10in}
\caption{Minimal uniton numbers for harmonic maps into exceptional symmetric spaces.}
\end{table}

\vspace{.25in}

\textbf{Acknowledgements:}
The second author would like to thank John Wood for helpful conversations. He also benefited from clarifying correspondence with Francis Burstall.

\end{document}